\documentclass{amsart}
\usepackage[latin1]{inputenc}
\usepackage{amsmath}
\usepackage{amsthm}
\usepackage{amsfonts}
\usepackage{amssymb}
\usepackage[T1]{fontenc}
\usepackage{float}
\usepackage[all]{xy}
\usepackage{tikz}
\usepackage{bbm}
\usepackage{color}
\usepackage{hyperref}
\usepackage{stmaryrd}

\theoremstyle{plain}
\newtheorem{theorem}{Theorem}[section]
\newtheorem*{thm-nonumber}{Theorem}
\newtheorem{prop}[theorem]{Proposition}
\newtheorem{lem}[theorem]{Lemma}

\theoremstyle{definition}
\newtheorem{defi}[theorem]{Definition}
\newtheorem{notation}[theorem]{Notation}
\newtheorem{rmq}[theorem]{Remark}
\newtheorem{exmp}[theorem]{Example}

\def\<{\left<}
\def\>{\right>}

\def\k{\mathbf{k}}

\def\ens#1{\left\{ #1 \right\}}

\def\fl{{\longrightarrow}\,}
\def\lf{{\longleftarrow}\,}

\def\II{{\mathbf{Ind}}}

\def\A{{\mathbf{A}}}
\def\s{{\mathbf{S}}}

\def\E{{\mathbf{E}}}

\def\b#1{\overline{#1}}

\def\Z{{\mathbb{Z}}}

\def\1{\mathbbm{1}}

\def\x{{\mathbf x}}

\def\ZZ{\mathbb{Z}}

\def\ens#1{\left\{ #1 \right\}}

\def\rep{\rm{rep}}
\def\Rep{\rm{Rep}}

\def\min{{\rm{min}\,}}

\def\Aii{_{\infty}{\mathbb A}_{\infty}}
\def\IG{C_{\infty}}
\def\IGC{\overline{C_{\infty}}}

\def\1{\mathbbm{1}}

\def\Prufer#1#2#3#4{
	\coordinate (x) at (#1,#2);
	\fill (x) circle (.1);
	\draw[#4] (x) .. controls (#1,#2+2.5) and (#1+.5,#2+3) .. (#1+#3,#2+3);
}

\def\adic#1#2#3#4{
	\coordinate (x) at (#1,#2);
	\fill (x) circle (.1);
	\draw[#4] (x) .. controls (#1,#2+2.5) and (#1-.5,#2+3) .. (#1-#3,#2+3);
}

\title[Transfinite mutations]{Transfinite mutations in the completed infinity-gon} 
\author{Karin~Baur and Sira~Gratz}

\begin{document}

\begin{abstract}
 We introduce mutation along infinite admissible sequences for infinitely marked surfaces, that is surfaces with infinitely many marked points on the boundary. We show that mutation along such admissible sequences produces a preorder on the set of triangulations of a fixed infinitely marked surface. We provide a complete classification of the strong mutation equivalence classes of triangulations of the infinity-gon and the completed infinity-gon respectively, where strong mutation equivalence is the equivalence relation induced by this preorder. Finally, we introduce the notion of transfinite mutations in the completed infinity-gon and show that all its triangulations are transfinitely mutation equivalent, that is we can reach any triangulation of the completed infinity-gon from any other triangulation via a transfinite mutation. 
\end{abstract}

\maketitle

\setcounter{tocdepth}{1}
\tableofcontents

\section*{Comments}

\section*{Introduction}

Triangulations of surfaces with marked points give rise to an interesting class of cluster algebras, which are tractable but sufficiently complicated to display a rich array of cluster combinatorics. The fact that they come together with a natural topological model means they play a key role in advancing our understanding of cluster theory, serving as important examples to test theories about general cluster
algebras and categories (cf.\ for example \cite{BZ:markedsurfaces} \cite{FST:markedsurfaces}, \cite{HJ:Ainfinite} \cite{MSW:positivitysurfaces} and \cite{MSW:bases}). Traditionally, only triangulations of surfaces with finitely many marked points have been studied in the context of cluster theory. With the rising interest in cluster algebras and categories of infinite rank (cf.\ for example \cite{GG:infinite}, \cite{G:colimits}, \cite{HL:KRmodules} \cite{HJ:Ainfinite}, \cite{IT:cyclicposets}), it is natural to extend the theory to a setting with infinitely many marked points, and consider what we call {\em infinitely marked surfaces}.

The idea to consider triangulations and mutations of infinitely marked surfaces is not new and has been executed in the context of cluster categories for example in \cite{HJ:Ainfinite} and \cite{IT:cyclicposets} and in the context of cluster algebras in \cite{GG:infinite} and \cite{G:colimits}. By introducing infinitely many marked points, interesting phenomena occur which do not appear in the finite setting. One noteable feature of infinitely marked surfaces, as opposed to finitely marked surfaces, is that two different triangulations are in general not connected by  finitely many mutations. In particular, two distinct triangulations of the same infinitely marked surface will in general give rise to two distinct cluster algebras of infinite rank in the sense of \cite{GG:infinite}.

In the present paper we study infinite mutations for infinitely marked surfaces, motivated by overcoming the finiteness constraints of the classical theory. We introduce the notion of mutation along infinite admissible sequences, and show such mutations connect previously disconnected components of the exchange graph. In fact, examples of mutations in cluster algebras along infinite admissible sequences have previously been used in \cite{HL:KRmodules}, and we formalize the idea here in the context of infinitely marked surfaces. We consider two important examples in more detail: the $\infty$-gon, which can be pictured as the line of integers, and the completed $\infty$-gon, which we obtain from the $\infty$-gon by completing with points at $\pm \infty$. Our main reason for studying these examples is that they provide combinatorial models for relatively well-studied examples in cluster theory of infinite rank: the $\infty$-gon relates to the cluster category studied in \cite{HJ:Ainfinite}, and in more generality in \cite{IT:cyclicposets}, and to the cluster algebras studied in \cite{GG:infinite}, whereas the completed $\infty$-gon 
relates to the representation theory of a polynomial ring in one variable. 

Single mutations are involutive, and therefore, if we can mutate from a triangulation $T$ to a triangulation $T'$ in finitely many steps, there is a way to mutate back from $T'$ to $T$. This is not the case anymore if we consider mutation along infinite admissible sequences. In this sense, we can think of mutations along infinite admissible sequences as being directed. We write $T \leq_s T'$ whenever we can mutate $T$ to $T'$ along an admissible sequence.  Our first main result is the following theorem.

\begin{thm-nonumber}[Theorem~\ref{T:partial order}]
The relation $\leq_{s}$ defines a preorder on the set of triangulations of a fixed infinitely marked surface.
\end{thm-nonumber}

The main focus of our article is on the $\infty$-gon and on the completed $\infty$-gon, two surfaces 
that have been at the centre of interest in cluster theory because of their relation to Dynkin type $A$ 
combinatorics.  We call two triangulations $T$ and $T'$ strongly mutation equivalent, if we have $T \leq_s T'$ as well as $T' \leq_s T$, that is, if they are equivalent under the equivalence relation induced by the above preorder.

\begin{thm-nonumber}[Theorems~\ref{T:strong equiv.cl.} and~\ref{T:strong equiv.cl.completed}]
 Two triangulations of the $\infty$-gon are strongly mutation equivalent if and only if they are both locally finite or they both have a left fountain at $a \in \ZZ$ and a right fountain at $b \in \ZZ$ with $a \leq b$.
 
 Two triangulations of the completed $\infty$-gon are strongly mutation equivalent if and only if they are both locally finite or they both have a left fountain at $a \in \ZZ \cup \{\pm \infty\}$ and a right fountain at $b \in \ZZ \cup \{\pm \infty\}$, where $a \leq b$ or $a = \infty$ and $b \in \ZZ$ or $a \in \ZZ$ and $b = -\infty$.
\end{thm-nonumber}

Mutating a triangulation along an infinite admissible sequence does not in general yield a triangulation. As a next step, we introduce completed mutations in the completed $\infty$-gon. In general, there are many ways to complete what one obtains from such a mutation to a triangulation. In the setting of the completed $\infty$-gon, there is however a natural way to complete with strictly asymptotic arcs, that is, with arcs connecting to the limit points at $\pm \infty$.

Finally, we introduce transfinite mutations in the completed $\infty$-gon. They are mutations along possibly infinite sequences of completed mutations. We call two triangulations $T$ and $T'$ transfinitely mutation equivalent, if there exists a transfinite mutation from $T$ to $T'$ as well as one from $T'$ to $T$.

\begin{thm-nonumber}[Theorem~\ref{T:all-transf-equiv}] 
 Any two triangulations of the completed $\infty$-gon are transfinitely mutation equivalent.
\end{thm-nonumber}

During the completion of this paper we learned that I. Canakci and A. Felikson are independently studying infinite 
sequences of mutations for cluster algebras coming from infinitely marked surfaces. 
Their results are now available as a preprint \cite{cf}. 

\section{Triangulations of infinitely marked surfaces}\label{section:infinity} 

Throughout this paper we only consider surfaces with boundary with marking such that all the marked points lie on the boundary. We are however convinced that the theory presented in this manuscript can be naturally extended to allow punctures, i.e.\ internal marked points. Throughout, when we speak of an infinitely marked surface $(S,M)$, we mean the following setup.

\begin{defi}
 A {\em infinitely marked surface} is a pair $(S,M)$ where 
 \begin{itemize}
  \item $S$ is a connected oriented 2-dimensional Riemann surface with a non-trivial boundary $\delta S$;
  \item$M \subseteq \delta S$ is an infinite
  set of marked points such that each connected component of $\delta S$ contains at least one marked point in $M$. 
 \end{itemize}
 
\end{defi}

Throughout the paper we denote by $(S,M)$ an infinitely marked surface. We define arcs in an infinitely marked surface analogously to \cite[Definition~2.2]{FST:markedsurfaces}.

\begin{defi}
 An {\em arc $\theta$ in $(S,M)$} is a curve such that
 \begin{enumerate}
  \item the curve $\theta$ connects two marked points in $M$,
  \item the curve $\theta$ does not intersect itself, except possibly at its endpoints,
  \item except for the endpoints, the curve $\theta$ is distinct from $\delta S$,
  \item the curve $\theta$ is not isotopic to a connected component of $\delta S \setminus M$.
 \end{enumerate}
 Arcs are considered up to isotopy inside the class of such curves.
 On the other hand, if a curve satisfies $(1)$, $(2)$ and $(3)$, but it is isotopic to a connected component of $\delta S \setminus M$, we call it an {\em edge} of $(S,M)$. We set
 \[
  \A(S,M) = \{\text{arcs of $(S,M)$}\} \; \; \text{and} \;\; \E(S,M) = \{\text{edges of $(S,M)$}\}.
 \]
 We say that two arcs $\alpha \neq \beta \in \A(S,M)$ are {\em compatible}, if they do not intersect in $S \setminus M$.
\end{defi}

 We are interested in triangulations of $(S,M)$, which traditionally correspond to clusters in the theory of cluster algebras and cluster categories.
 
 \begin{defi}\label{D:triangulation}
  A {\em triangulation} of $(S,M)$ is a maximal set of compatible arcs of $(S,M)$. 
 \end{defi}
 
 \begin{rmq}
  The term ``triangulation'' can in some instances be seen as an abuse of language. In fact, a triangulation in the sense of Definition \ref{D:triangulation} does not in general partition the surface $(S,M)$ into triangles. In the case where $M$ is not discrete we might even get rather unintuitive triangulations, for example if we consider the unit disc where every point on its boundary $S^1$ is a marked point. Picking a point $a \in S^1$, the set
  \[
   \{(a,b) \mid b \in S^1 \setminus \{a\}\}
  \]
  is a triangulation of this infinitely marked surface in the sense of Definition \ref{D:triangulation}. 
  
 \end{rmq}

In this paper, we are interested in two particular marked surfaces, which are closely related to Dynkin type $A$ combinatorics. In many ways they form the simplest cases of infinitely marked surfaces. The first is the $\infty$-gon $\IG$, which has a discrete set of marked points, and the second is the completed $\infty$-gon $\IGC$, which we get from $\IG$ by adding limit points.

\subsection{Triangulations of the infinity-gon}

Consider the infinitely marked surface $\IG = (S,M)$, where $S = D^1$ is the unit disc and $M \subseteq S^1$ is a discrete set of marked points with one two-sided limit point. We call $\IG$ the {\em $\infty$-gon} and cut open the circle at the limit point to picture the boundary $\delta S$ as the line of integers, cf.\ Figure \ref{fig:infty-gon}. An {\em arc} in $\IG$ is an ordered pair of integers $(i,j)$ with $i \leq j-2$. Two arcs $(i,j)$ and $(k,l)$ in $C_{\infty}$ are not compatible if and only if we have $i < k < j < l$ or $k < i < l < j$. 

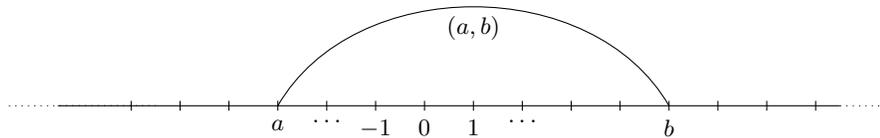
\begin{figure}
\begin{center}
\begin{tikzpicture}[scale =.65]
\tikzstyle{every node}=[font=\small]
\draw (-7.5,0) -- (8.5,0);
\draw[dotted] (-8.5,0) -- (-5.5,0);
\draw[dotted] (8.5,0) -- (9.5,0);
\draw (0,0.1) -- (0,-0.1) node[below]{$0$};
\draw (1,0.1) -- (1,-0.1) node[below]{$1$};
\draw (-1,0.1) -- (-1,-0.1) node[below]{$-1$};
\draw (2,0.1) -- (2,-0.1) node[below]{$\ldots$};
\draw (-2,0.1) -- (-2,-0.1) node[below]{$\ldots$};
\draw (3,0.1) -- (3,-0.1); 
\draw (4,0.1) -- (4,-0.1); 
\draw (-3,0.1) -- (-3,-0.1) node[below]{$a$};
\draw (5,0.1) -- (5,-0.1) node[below]{$b$};
\draw (6,0.1) -- (6,-0.1);
\draw (7,0.1) -- (7,-0.1);
\draw (8,0.1) -- (8,-0.1);
\draw (-4,0.1) -- (-4,-0.1);
\draw (-5,0.1) -- (-5,-0.1);
\draw (-6,0.1) -- (-6,-0.1);

\path (-3,0) edge [out= 60, in= 120] node[below]{$(a,b)$} (5,0);

\end{tikzpicture}
\end{center}
\caption{An arc $(a,b)$ in the $\infty$-gon $\IG$}\label{fig:infty-gon}
\end{figure}

\begin{defi}
We call a triangulation of $\IG$ {\em locally finite} if for any vertex $l \in \ZZ$ there are only finitely many arcs of the form $(k,l)$ or $(l,m)$.

Let $n \in \Z$. A \emph{left fountain at $n$} is a family $\ens{(p,n) \ | \ p \in P}$ in $\A(\IG)$ such that $P \subseteq \; ]-\infty,n-2]$ is infinite. Dually, a \emph{right fountain at $n$} is a family $\ens{(n,p) \ | \ p \in P}$ in $\A(\IG)$ such that $P \subseteq [n+2,+\infty[$ is infinite. A \emph{split fountain} is the union of a left fountain at some $n \in \ZZ$ and of a right fountain at some $m \in \ZZ$, with $m>n$. 
		
\end{defi}

\begin{lem}[{\cite[Lemma~3.3]{HJ:weakclustertilting}}]
 A triangulation $T$ of $\IG$ is either locally finite or has precisely one left fountain and one right fountain.
\end{lem}

\subsection{Triangulations of the completed infinity-gon}

We complete the $\infty$-gon with a point at $\infty$ and a point at $-\infty$. This yields the {\em completed $\infty$-gon $\IGC$}. Formally speaking, it is the unit disc with countably many marked points on the boundary, which converge to to a limit point $a$ in a clockwise direction and to a limit point $b$ in an anti-clockwise direction, and where there are no marked points between $a$ and $b$ when going in a clockwise direction. Cutting open the circle at a point between $a$ and $b$ (in a clockwise direction), we obtain the line of integers with two added limit points at $\pm \infty$, cf.\ Figure \ref{fig:arcsInf}.

\begin{rmq}\label{R:one limit point}
 In some ways it might be more natural to consider the completion where we only add in one point at $\infty$ (and consider the unit disc with countably many marked points that converge to precisely one limit point from both a clockwise and an anti clockwise direction). However, we are particularly interested in the combinatorial model with both points at $\infty$ and $-\infty$ as it fits well with the combinatorics of modules over polynomial rings, cf. Section~\ref{ssec:poly-rings}. 
\end{rmq}

Arcs in $\overline{C_\infty}$ come in two different forms. They can be of the form $(i,j)$ where $i,j$ are integers with $i \leq j-2$. Such an arc $(i,j)$ is called a {\em peripheral arc}. Furthermore, we get the following {\em strictly asymptotic arcs} involving the points at $\pm \infty$:
\begin{itemize}
\item For each $m \in \ZZ$ the {\em adic curve at $i$} is the arc $\alpha_m = (-\infty, m)$.
\item For each $m \in \ZZ$ the {\em Pr\"ufer curve at $m$} is the arc $\pi_m = (m, \infty)$. 
\item The {\em generic curve} is the arc $z = (-\infty,\infty)$.
\end{itemize}
Figure \ref{fig:arcsInf} provides a picture of some strictly asymptotic arcs.
Two arcs $(i,j)$ and $(k,l)$ in $\IGC$ are not compatible if and only if we have $i < k < j < l$ or $k < i < l < j$.

	\begin{figure}[htb]
		\begin{center}
			\begin{tikzpicture}[scale = .4]
				\tikzstyle{every node} = [font = \small]
				\foreach \x in {0}
				{
					\foreach \y in {-8}
					{
						\draw[-] (\x-5,\y-4) -- (\x+5,\y-4);

						\foreach \t in {-4,-3,...,4}
						{
							\fill (\x+\t,\y-4) circle (.1);
						}

						\fill (\x-4,\y-4) ;
						\fill (\x-3,\y-4) ;
						\fill (\x-2,\y-4) ;
						\fill (\x-1,\y-4) node [below] {$s$};
						\fill (\x,\y-4) ;
						\fill (\x+1,\y-4) ;
						\fill (\x+2,\y-4) node [below] {$r$};
						\fill (\x+3,\y-4) ;
						\fill (\x+4,\y-4) ;

						\adic{\x-1}{\y-4}{9}{}
						\fill (\x-3,\y-1.5) node [below] {$\alpha_{s}$};
						\fill (\x-10,\y-1) node [below] {\tiny $-\infty$};

						\Prufer{\x+2}{\y-4}{9}{}
						\fill (\x+4,\y-1.5) node [below] {$\pi_{r}$};
						\fill (\x+11,\y-1) node [below] {\tiny $+\infty$};

						\draw[-] (\x-10,\y) -- (\x+11,\y);
						\fill (\x+.5,\y) node [above] {$z$};

					}
				}
			\end{tikzpicture}
		\end{center}
		\caption{Arcs in the completed $\infty$-gon $\IGC$}\label{fig:arcsInf}
	\end{figure}

\begin{rmq}\label{R:generic}
The generic curve $z$ is compatible with any arc in $\IGC$. Therefore, any triangulation of $\IGC$ contains $z$. 
When we explicitly write down triangulations of $\IGC$, for brevity we will usually omit the generic curve.
\end{rmq}

The notions of local finiteness, right, left and split fountains naturally carry over from triangulations of $\IG$. However, we can also have fountains in $\IGC$ at $\pm \infty$.

\begin{defi}

A {\em left fountain} at $\infty$ (respectively at $- \infty$) is a family $\ens{(p,\infty) \ | \ p \in P}$ (respectively a family $\ens{(-\infty,p) \ | \ p \in P}$) in $\A(\IGC)$  where $P \cap \; ]-\infty,0]$ is infinite. Dually, a {\em right fountain} at $\infty$ (respectively at $- \infty$) is a family $\ens{(p,\infty) \ | \ p \in P}$ (respectively a family $\ens{(-\infty,p) \ | \ p \in P}$) in $\A(\IGC)$  where $P \cap [0,\infty[$ is infinite.
		
\end{defi}

\begin{lem}\label{lem:fountaincurve}
	Let $T$ be a triangulation of $\IGC$ and $n \in \Z$. If $T$ contains a right fountain at $n$, then $\pi_n \in T$ and if $T$ contains a left fountain at $n$, then $\alpha_n \in T$.
\end{lem}

\begin{proof}
Assume that $T$ contains a right fountain at $n$ and consider the Pr\"ufer curve $\pi_n$. 
It is compatible with any arc in $\IGC$ except~:
\begin{itemize}
	\item the arcs of the form $(m,l)$ with $m<n$ and $l>n$,
	\item the adic curves of the form $\alpha_l$ with $l>n$.
\end{itemize}
Let $l>n$. Since $T$ contains a right fountain, it contains an arc of the form $(n,p)$ with $p>l$. Therefore, $(n,p)$ intersects any arc of the above form. Thus, $T$ only contains arcs which are compatible with $\pi_n$. By maximality of $T$, we have $\pi_n \in T$.
	
The fact that if $T$ contains a left fountain at $n$ then $\alpha_n \in T$ follows by symmetry.
\end{proof}
		
		\begin{lem}\label{lem:pushitfurther}
			Let $T$ be a triangulation of $\IGC$ and $n \in \Z$. 
			\begin{enumerate}
				\item Assume that $\pi_n \in T$. Then $T$ contains a right fountain at $n$ or there exists an $m>n$ such that $\pi_m \in T$.
				Dually, if $\alpha_n \in T$, then $T$ contains a left fountain at $n$ or there exists $m<n$ such that $\alpha_m \in T$.
				\item Assume that $\pi_n \in T$. If $T$ contains a right fountain at $n$ and there is no $k<n$ with $\pi_k\in T$, then $\alpha_n\in T$. 
				Dually, if $\alpha_n \in T$ and $T$ contains  a left fountain at $n$, such that there is no $l>n$ with $\alpha_l\in T$, then $\pi_n\in T$. 
			\end{enumerate}
		\end{lem}
		\begin{proof}
			We only prove the statements for the Pr\"ufer curves, the ones for the adic curves being dual.
			
			(1) Assume that $T$ contains the Pr\"ufer curve $\pi_n$ and assume that there are at most finitely many arcs of the form $(n,p)$ with $p \geq n+2$. Let 
			\[
			p_0 = \max\ens{p \geq n+1 \ | \ (n,p) \in T \cup \E(\IGC)}.
			\]
			If $T$ contains a right fountain at $p_0$, then it follows from Lemma \ref{lem:fountaincurve} that $\pi_{p_0} \in T$ and we are done. Assume therefore that $T$ does not 					contain a right fountain at $p_0$ and let
			\[
			 p_1 = \max\ens{p \geq p_0+1 \ | \ (p_0,p) \in T\cup \E(\IGC)}.
			\]
			Then $(n,p_1)$ is compatible with any arc in $T$, so it belongs to $T$. However, $p_1 > p_0$, a contradiction. 
			(2) is clear: the adic curve $\alpha_n$ is compatible with $\pi_n$ and every arc to the right of $n$ as well as every arc to the left of $n$. 
		\end{proof}
		
We obtain the following classification of triangulations of $\IGC$.		

\begin{theorem}\label{T:classification triangulations of IGC}
Let $T$ be a triangulation of $\IGC$. Then exactly one of the following holds.
\begin{itemize}
  \item $T$ is locally finite and consists exclusively of peripheral arcs and the generic curve.
  \item $T$ has a left fountain at a unique $a \in \ZZ \cup \{\pm \infty\}$ and a right fountain at a unique $b \in \ZZ \cup \{\pm \infty\}$ with $a \leq b$.
  \item $T$ has a left fountain at a unique $a \in \ZZ \cup \{\pm \infty\}$ and a right fountain at a unique $b \in \ZZ \cup \{\pm \infty\}$ with $a = \infty$ and $b \in \ZZ$ or $a \in \ZZ$ and $b = -\infty$. 
\end{itemize}
\end{theorem}

\begin{proof}

Assume $T$ is locally finite and assume as a contradiction that $T$ contains a strictly asymptotic arc (that is not the generic curve). Without loss of generality assume $\pi_m \in T$, for some $m \in \mathbb{Z}$. By Lemma \ref{lem:pushitfurther}(1) this implies that it contains a right fountain at an integer $n \geq m$ or infinitely many Pr\"ufer curves, contradicting the assumption of local finiteness.

Assume now that $T$ is not locally finite, thus it contains at least one left or right fountain. It is clear that $T$ cannot contain two right (respectively left) fountains at $b \neq b'$ since they would intersect
at infinitely many arcs close to $\infty$ (respectively close to $-\infty$). Thus $T$ has a left fountain
at at most one $a \in \mathbb{Z} \cup \{\pm \infty\}$ and a right fountain at at most one $b \in \mathbb{Z} \cup \{\pm \infty\}$. 
		
Assume $T$ has a right fountain at $b \in \mathbb{Z} \cup \{\pm \infty\}$. If $b \neq -\infty$ this means that $T$ contains a Pr\"ufer $\pi_m$ at some $m \leq b$.
By Lemma \ref{lem:pushitfurther}(2) this implies that it has a left fountain at $\infty$ or it contains an adic $\alpha_n$ for some $n \leq m$. In the latter case,	
Lemma \ref{lem:pushitfurther}(1) implies that it contains a left fountain at some $a \leq n$. On the other hand, if $b = -\infty$, then $T$ contains infinitely many adics
of the form $\alpha_n$ for $n \geq 0$. By Lemma \ref{lem:pushitfurther}(1), $T$ contains a left fountain.

It follows by symmetry that if $T$ contains a left fountain, then it contains a right fountain. Therefore, every triangulation of $\IGC$ is either purely
peripheral or it contains a left fountain at a unique $a \in \mathbb{Z} \cup \{\pm \infty\}$ and a right fountain at a unique 
$b \in \mathbb{Z} \cup \{\pm \infty\}$. If $a \neq \infty$ it follows that $b \geq a$ or $b = -\infty$ since otherwise infinitely many arcs in the right fountain at 
$b$ would intersect infinitely many arcs in the left fountain at $a$. 
If $a = \infty$, then with the same argument we must have $b \neq -\infty$. 

\end{proof}

\subsection{On the combinatorics of modules over a polynomial ring} \label{ssec:poly-rings}
The reason we are particularly interested in triangulations of $\IGC$ is the connection between the combinatorial structure of $\IGC$ and the combinatorial structure of 
the indecomposable objects of the category $\Rep(\Aii)$ of representations over an algebraically closed field 
$\k$ of the quiver 
$$\overleftarrow{\mathbb{Z}}=\Aii: \cdots \lf -2 \lf -1 \lf 0 \lf 1 \lf 2 \lf \cdots$$
whose vertices are labelled by the integers and where there are arrows $i-1 \lf i$ for any $i \in \Z$.
		
We denote by $\rep(\Aii)$ the full subcategory of $\Rep(\Aii)$ formed by the finite-dimensional representations. An indecomposable object in $\rep(\Aii)$ is isomorphic to a representation 
of the form $M_{ij}$ with $i,j \in \Z$ and $i \leq j-2$ where $M_{ij}$ is one-dimensional at each of the vertices $i+1, \ldots, j-1$ and where all the maps between non-zero vector spaces are the identity.

For any $i \in \Z$, we have injections
\[
 M_{i,i+2} \hookrightarrow M_{i,i+3} \hookrightarrow M_{i,i+4} \hookrightarrow \ldots
\]
The colimit of this system is the indecomposable representation $\Pi_i \in \Rep(\Aii)$ which is one-dimensional at each of the vertices in $[i+1,+\infty[$ and where all the maps between non-zero vector spaces are identities. The representation $\Pi_i$ is called the \emph{Pr\"ufer module} at vertex $i$.
		
Dually, for any $i \in \Z$, we have surjections
\[
 \ldots \twoheadrightarrow M_{i-4,i} \twoheadrightarrow M_{i-3,i} \twoheadrightarrow M_{i-2,i}.
\]
The limit of this system is the indecomposable representation $A_i \in \Rep(\Aii)$ which is one-dimensional at each of the vertices in $]-\infty,i-1]$ and where all the maps between non-zero vector spaces are identities. The representation $A_i$ is called the \emph{adic module} at vertex $n$.
		
We also have surjections 
$$\cdots \twoheadrightarrow \Pi_{i-1} \twoheadrightarrow \Pi_{i} \twoheadrightarrow \Pi_{i+1} \twoheadrightarrow \cdots$$
The limit of this system is the indecomposable representation $G \in \Rep(\Aii)$ which is one-dimensional at each vertex in $\Z$ and where all the maps between non-zero vector spaces are identities. The representation $G$ is called the \emph{generic module}.

We denote by $\II$ the set of (isomorphism classes of) indecomposable finite-dimensional representations of $\Aii$ together with the indecomposable Pr\"ufer, adic and generic modules. Then there is a natural bijection 
\[
 \Phi: \left\{\begin{array}{rcll}
	\b\A(\IG) & \fl & \II \\
	(i,j) & \mapsto & M_{ij} & \text{ for any } i \leq j-2 \in \Z~;\\
	\pi_i & \mapsto & \Pi_i & \text{ for any } i \in \Z~;\\
	\alpha_i & \mapsto & A_i & \text{ for any } i \in \Z~;\\
	z & \mapsto & G.
\end{array}\right.
\]

\begin{rmq}
Under the bijection $\Phi$, triangulations of $\IGC$ correspond to maximal rigid subcategories of $\Rep(\Aii)$. 
This follows from \cite[Section 5]{BBM:torsiontubes} and the observation that the generic curve $z$ is compatible with any other curve. 
\end{rmq}

\section{Mutations of triangulations}

At the heart of cluster combinatorics arising from triangulations of the $\infty$-gon $C_{\infty}$ lies the concept of mutation.

\begin{defi}
Let $T$ be a triangulation of an infinitely marked surface $(S,M)$. 
We say that an arc $\theta \in T$ is {\em mutable} if and only if there exists an arc $\theta' \neq \theta$ in $\A(S,M)$
such that 
\[
  \mu^T_{\theta}(T)= (T \setminus \{\theta\}) \cup \{\theta'\}
\]
is a triangulation of $(S,M)$. We call $\mu^T_{\theta}(T)$ the {\em mutation of $T$ at $\theta$}. We will use the following notation: For $\gamma \in T$ we set
\[
  \mu^T_{\theta}(\gamma) = \begin{cases}
			  \gamma \; \text{if} \; \gamma \neq \theta\\
			  \theta' \; \text{if}\; \gamma = \theta.
			\end{cases}
\]
Usually, the triangulation in which we mutate will be clear from context and we will omit the superscript and just write $\mu_{\theta}(T)$ and $\mu_{\theta}(\gamma)$ for $\mu^T_{\theta}(T)$ and $\mu^T_{\theta}(\gamma)$ respectively.

\end{defi}

\begin{rmq}\label{R:mutable}
Let $T$ be a triangulation of $(S,M)$.
It is straightforward to check that an arc $\gamma \in T$ is mutable if and only if $\gamma$ is a diagonal in a quadrilateral with edges in $T \cup \E(S,M)$ (cf.\ also \cite[Section~3]{FST:markedsurfaces}),
and that its mutation is given by the other diagonal $\gamma' \neq \gamma$ in the quadrilateral. We call the set 
\[
S(\gamma) = \{\text{sides of the quadrilateral with diagonal $\gamma$}\} \cap \A(S,M) \subseteq T
\]
the {\em quadrilateral in $T$ with diagonal $\gamma$}. 

By abuse of notation we will more generally call an arc $\gamma$ in a set $N$ of compatible arcs of $(S,M)$
mutable, if $S(\gamma) \subseteq N$. With the notations as above we write $\mu_\gamma(N) = (N \setminus \{\gamma\}) \cup \gamma'$.

\end{rmq}

Note that if $\alpha$ and $\beta$ are mutable arcs in a triangulation $T$ of $(S,M)$ then
$\alpha \notin \{\beta\} \cup S(\beta)$ if and only if $\beta \notin \{\alpha\} \cup S(\alpha)$. The following lemma will be useful throughout the paper.

\begin{lem}\label{L:commutativity}
Let $T$ be a triangulation of $(S,M)$.
Let $\alpha, \beta \in T$ be mutable and assume $\alpha \notin \{\beta\} \cup S(\beta)$. Then $\alpha$ is mutable in $\mu_\beta(T)$ and $\beta$ is mutable in $\mu_{\alpha}(T)$ and for all $\gamma \in T$ we have
\[
  \mu_\beta \circ \mu_\alpha(\gamma) = \mu_\alpha \circ \mu_\beta(\gamma).
\]

\end{lem}

\begin{proof}
 Since $\alpha \in T$ and $S(\alpha) \subseteq T$ and $\beta \notin \{\alpha\} \cup S(\alpha)$ we have $\gamma = \mu^T_{\beta}(\gamma) \in \mu_\beta(T)$ for all $\gamma \in \{\alpha\} \cup S(\alpha)$. In particular, $\alpha \in \mu^T_\beta(T)$ is mutable. Analogously, $\beta \in \mu^T_\alpha(T)$ is mutable. Let $\alpha' \neq \alpha$ be the other diagonal in $S(\alpha)$ and let $\beta' \neq \beta$ be the other diagonal in $S(\beta)$. Because $S(\alpha) \subseteq \mu^T_\beta(T)$ we have $\mu^{\mu^T_\beta(T)}_\alpha(\alpha) = \alpha' =  \mu^T_\alpha(\alpha) $ and since $S(\beta) \subseteq \mu^T_\alpha(T)$ we have $\mu^{\mu^T_\alpha(T)}_\beta(\beta) = \beta'=  \mu^T_\beta(\beta)$. Since $\beta$ and $\beta'$ intersect, but $\alpha$ and $\beta$ do not, we have $\alpha \neq \beta'$ and analogously $\beta \neq \alpha'$. Therefore we obtain 
 
 \[
  \mu_\alpha \circ \mu_\beta(\alpha) = \mu^{\mu^T_{\beta}(T)}_\alpha(\alpha) = \alpha' = \mu^{\mu^T_{\alpha}(T)}_\beta(\alpha') = \mu_\beta \circ \mu_\alpha(\alpha)
 \]
 and symmetrically
 \[
  \mu_\alpha \circ \mu_\beta(\beta) = \mu_\beta \circ \mu_\alpha(\beta).
 \]
 Clearly, for all $\gamma \in T \setminus \{\alpha,\beta\}$ we have $\mu_\alpha \circ \mu_\beta (\gamma) = \gamma = \mu_\beta \circ \mu_\alpha(\gamma)$, which proves the claim.

\end{proof}

\begin{lem}\label{L:mutable1}
 Assume $T$ is a triangulation of $(S,M)$ and $\alpha \in T$ is mutable. Then $\gamma$ is mutable in $T$ if and only if $\mu_{\alpha}(\gamma)$ is mutable in $\mu_{\alpha}(T)$.
\end{lem}

\begin{proof}
 The statement is clear if $\gamma = \alpha$. Assume thus $\gamma \neq \alpha$ and let $S(\gamma)$ be the quadrilateral in $T$ with diagonal $\gamma$. If  $\alpha \notin S(\gamma) \subseteq T$, then we still have $\{\gamma\} \cup S(\gamma) \subseteq T'$, and $\gamma$ is mutable. Otherwise, if $\alpha \in S(\gamma)$, then $\alpha$ and $\gamma$ are the sides of a common triangle with sides $\alpha, \beta, \gamma$ in $T \cup \E(S,M)$ and $S(\gamma) = \{\alpha,\beta,\delta,\epsilon\} \cap \A(S,M)$ for some $\delta, \epsilon \in T \cup \E(S,M)$. Since $\alpha'$ is still a diagonal in $S(\alpha) \subseteq \mu_{\alpha}(T)$ and $\gamma = \mu_{\alpha}(\gamma) \in S(\alpha)$, the arcs $\alpha'$ and $\mu_{\alpha}(\gamma)$ are sides of a common triangle with sides $\alpha', \beta', \gamma$ in $\mu_{\alpha}(T) \cup \E(S,M)$ and we have a quadrilateral $S(\mu_{\alpha}(\gamma)) = \{\alpha', \beta', \delta, \epsilon\} \cap \A(S,M)$ in $\mu_{\alpha}(T)$ with diagonal $\mu_{\alpha}(\gamma)$. 
\end{proof}

\subsection{Mutations in the infinity-gon} If $T$ is a triangulation of $\IG$ either all arcs or all arcs but one are mutable. 

\begin{defi}
 Let $T$ be a triangulation of $\IG$. We say that an arc $(a,b) \in T$ {\em connects a split fountain} if there is a left fountain at $a$ and a right fountain at $b$ in $T$.
\end{defi}

	\begin{prop}\label{prop:mutIG}
		Let $T$ be a triangulation of $\IG$ and let $\theta \in T$. Then $\theta$ is mutable if and only if does not connect a split fountain. 
	\end{prop}
	\begin{proof}
	If $T$ is locally finite or if it has a right and a left fountain at some $a \in \ZZ$, then by \cite[Lemmas~3.4 and~3.6]{HJ:weakclustertilting}, every arc is mutable. 
	
	On the other hand assume that $T$ has a split fountain, with a left fountain at $a \in \ZZ$ and a right fountain at $b\in \ZZ$. We show that the arc $(a,b)$ is the only non-mutable arc. Indeed, it is not mutable since every arc that intersects $(a,b)$ intersects infinitely many arcs in the right fountain at $b$ or the left fountain at $a$, therefore we cannot replace $(a,b)$ by another arc to obtain again a triangulation. We now show that every other arc is mutable: Every arc in $T \setminus \{(a,b)\}$ is of the form $(i,j) \neq (a,b)$ with $i < j \leq a$ or $b \leq i < j$ or $a \leq i < j \leq b$. If $i < j \leq a$, then there is an arc $(l,a) \in T$ with $l < i < j \leq a$, if $b \leq i < j$ then there is an arc $(b,k) \in T$ with $b \leq i < j < k$ and in the final case we have $a < i < j \leq b$ or $a \leq i < j < b$ with $(a,b) \in T$. In either case, it follows by \cite[Lemma~3.6]{HJ:weakclustertilting} that the arc $(i,j)$ is mutable.
	\end{proof}

\subsection{Mutations in the completed infinity-gon}
		
We will see that it is always possible to mutate triangulations of $\IGC$ at peripheral arcs. However, for strictly asymptotic arcs, the situation is slightly more complicated.

\begin{defi}[Arcs wrapping a fountain]
	Let $T$ be a triangulation of $\IGC$. We say that an arc $\gamma$ in $T$ is {\em wrapping a fountain in $T$} if $T$ contains a left (or right, respectively) 
	fountain at $m$ and $\gamma=\alpha_m$ (or $\gamma=\pi_m$ respectively). 
\end{defi}

\begin{prop}\label{prop:mutIGC}
	Let $T$ be a triangulation of $\IGC$ and let $\theta \in T$. Then $\theta$ is mutable if and only if $\theta$ is neither the generic curve nor wrapping a fountain in $T$.

\end{prop}
\begin{proof}
  Let $T$ be a triangulation of $\IGC$ and let $\theta \in T$. By Remark \ref{R:generic} and Lemma \ref{lem:fountaincurve} if $\theta$ is generic or wrapping a fountain then it is not mutable. 
  
	On the other hand assume $\theta \in T$ is not generic nor wrapping a fountain. Assume first that $\theta$ is strictly asymptotic. Without loss of generality, we assume that 
	$\theta = \alpha_m$ for some vertex $m \in \Z$ where there is no left fountain, the statement for a Pr\"ufer curve follows by symmetry. Then, it follows from Lemma \ref{lem:pushitfurther} that there is an adic arc $\alpha_n$ with 		$n<m$. We let 
	\[
	n_0 = \max \ens{n < m \ | \ \alpha_n \in T}.
	\]
	There are two possibilities. Either there is some $l>m$ such that $\alpha_l \in T$, in which case we set 
	\[
	n_1 = \min \ens{l > m \ | \ \alpha_l \in T}.
	\]
 Then as in the proof of  \cite[Proposition 1.6]{BD:compactifying}, $\theta' = (n_0,n_1)$ is the unique arc distinct from $\theta$ such that $T \setminus \ens{\theta} \sqcup \ens{\theta'}$ is a triangulation of $\IGC$. If there is no adic arc $\alpha_l$ with $l > m$, then $\pi_m$ does not intersect any arc in $T$ and thus $\pi_m \in T$. Therefore $\alpha_m$ is a diagonal in the quadrilateral $S(\alpha_m) = \{\alpha_{n_0}, \pi_m, (n_0,m), z\} \cap \A(\IGC)$ in $T$ and by Remark \ref{R:mutable} it is mutable.
  
  Assume now that $\theta = (i,j)$ is a peripheral arc, thus $-\infty < i < j < \infty$. If there is an arc $(a,b) \in T$ with $a \leq i < j < b$ or $a < i < j \leq b$, then it follows analogously to \cite[Lemma~3.6]{HJ:weakclustertilting} that $(i,j)$ is mutable. 
  
  On the other hand, assume there is no such arc. Then the arcs $\pi_i, \pi_j, \alpha_i, \alpha_j$ do not intersect any peripheral arcs in $T$. If $\pi_i \in T$, then there cannot be an adic $\alpha_k \in T$ with $k \geq i$ and therefore we also have $\pi_j \in T$. If, on the other hand, we have $\pi_i \notin T$, then there must exist an $l > i$ with $\alpha_l \in T$, and since $(i,j) \in T$ must not intersect $\alpha_l \in T$ we even have $l \geq j$. It follows that we cannot have any $\pi_k \in T$ with $k \leq l$ and therefore $\alpha_i$ and $\alpha_j$ do not intersect any strictly asymptotic arcs in $T$ either and thus $\alpha_i, \alpha_j \in T$. 
  
  Therefore we have $\pi_i, \pi_j \in T$ or $\alpha_i, \alpha_j \in T$, without loss of generality assume the former is the case (the latter case follows by symmetry). There exists a $k \in \ZZ$ with $i < k < j$ and $(i,k), (k,j) \in T \cup \E(\IGC)$ and the arc $(i,j)$ is a diagonal in the quadrilateral $S((i,j)) = \{(i,k), (k,j), \pi_i, \pi_j\} \cap \A(\IGC)$ in $T$. It follows by Remark \ref{R:mutable} that $(i,j) \in T$ is mutable.

\end{proof}

\section{Mutations along infinite admissible sequences}

Classically, the exchange graph of a marked surface (with finitely many marked points) is defined as the graph which has as vertices triangulations of the marked surface and as edges diagonal flips. In the finite setting, this exchange graph is connected, in the sense that for any two of its vertices there exists a finite path connecting them. However, if we extend this definition naively to infinitely marked surface, the resulting graph will not be connected anymore.
In particular, triangulations that have very similar structure are not necessarily connected by finite sequences of mutations. Consider for example the two locally finite triangulations
\[
	t_{lf} = \{(-k,k) \mid k \in \mathbb{Z}_{>0}\} \cup \{(-k,k+1) \mid k \in \mathbb{Z}_{>0}\}
\]
and
\[
	t^-_{lf} = \{(-k,k) \mid k \in \mathbb{Z}_{>0}\} \cup \{(-(k+1),k) \mid k \in \mathbb{Z}_{>0}\}
\]
of $\IG$.
They are both locally finite, thus seem to have very similar behaviour under mutation, however there exists no finite sequence of mutations from $t_{lf}$ to $t^-_{lf}$. We are however able to connect (these particular) triangulations via mutations, if we consider mutations along possibly infinite admissible sequences.

\begin{defi}\label{D:infinite mutation}
	Let $T$ be a triangulation of an infinitely marked surface $(S,M)$
	and let $I$ be a countable indexing set, for notational simplicity throughout this paper we take $I = \{1, \ldots, n\}$ if it is finite and $I = \mathbb{Z}_{> 0}$ if it is infinite. A sequence of arcs $\underline{\theta} = (\theta_i)_{i \in I}$ is called
	{\em $T$-admissible} if it satisfies the following:
	\begin{enumerate}
		\item{$\theta_1$ is mutable in $T$}
		\item{For all $1 \neq i \in I$, the arc $\theta_i$ is mutable in $\mu_{\theta_{i-1}} \circ \ldots \circ \mu_{\theta_1}(T)$.}
		\item{For all $\gamma \in T$ there exists an $l_\gamma \in I$ such that for all $k \geq l_\gamma$ we have
			\[
				\mu_{\theta_{k}} \circ \ldots \circ \mu_{\theta_1}(\gamma) = \mu_{\theta_{l_\gamma}} \circ \ldots \circ \mu_{\theta_1}(\gamma).
			\]
			}
	\end{enumerate}
For each arc $\gamma \in T$ we define the {\em mutation of $\gamma$ along $\underline{\theta}$} to be 
$\mu^T_{\underline{\theta}}(\gamma) = \mu_{\theta_{l_\gamma}} \circ \ldots \circ \mu_{\theta_1}(\gamma)$, where $l_{\gamma}$ is as in (3). We set 
	\[
		\mu_{\underline{\theta}}(T) = \{\mu^T_{\underline{\theta}}(\gamma) \mid \gamma \in T\}
	\]
and call it the {\em mutation of $T$ along $\underline{\theta}$.}
\end{defi}

If it is clear from context, we will usually omit the superscript and simply write $\mu_{\underline{\theta}}(\gamma)$ for $\mu^T_{\underline{\theta}}(\gamma)$.

\begin{exmp}\label{E:not a triangulation}
 The mutation of a triangulation $T$ along a $T$-admissible sequence is not necessarily a triangulation. Consider for example the triangulation
 \[
  t(0,0) = \{(0,k) \mid k \in \ZZ_{\geq 2}\} \cup \{(-k,0) \mid k \in \ZZ_{\geq 2}\} \cup \{\alpha_0\} \cup \{\pi_0\}
 \]
 of $\IGC$ and the $t(0,0)$-admissible sequence $\underline{\theta} = ((0,i))_{i \geq 2}$. We have 
 \[
  \mu_{\underline{\theta}}(t(0,0)) = \{(1,k) \mid k \in \ZZ_{\geq 3}\} \cup \{(-k,0) \mid k \in \ZZ_{\geq 2}\} \cup \{\alpha_0\} \cup \{\pi_0\}
 \]
 which is not a triangulation of $\IGC$: the arc $\pi_1$ does not intersect any arc in $\mu_{\underline{\theta}}(t(0,0))$, yet it is not contained in $\mu_{\underline{\theta}}(t(0,0))$.
\end{exmp}

\begin{rmq}\label{R:infinite mutation}

However, the mutation of a triangulation along a $T$-admissible sequence consists of mutually non-intersecting arcs: For any pair of arcs $\beta_1, \beta_2 \in T$ there exists a $k \in \mathbb{Z}$ such that 
$\mu_{\underline{\theta}}(\beta_i) = \mu_{\theta_k} \circ \ldots \circ \mu_{\theta_1}(\beta_i)$ for $i = 1,2$. Since $\mu_{\theta_k} \circ \ldots \circ \mu_{\theta_1}(T)$ is a triangulation, 
$\beta_1$ and $\beta_2$ do not intersect.

Moreover, $\mu_{\underline{\theta}}(T)$ always remains infinite: it follows directly from Definition \ref{D:infinite mutation} that $\mu_{\underline{\theta}}(\gamma) \neq \mu_{\underline{\theta}}(\gamma')$ for all $\gamma \neq \gamma' \in T$.
\end{rmq}

\begin{rmq}\label{R:no mutation}
 If $T$ and $T'$ are triangulations of $(S,M)$ and if there is an arc $\gamma \in T'$ that intersects infinitely many arcs in $T$, then there is no $T$-admissible sequence $\underline{\theta} = (\theta_i)_{i \in I}$ with $\mu_{\underline{\theta}}(T) = T'$.
 
 Indeed, if there were such a $T$-admissible sequence, then we would have an $i \in I$ such that $\gamma \in \mu_{\theta_i} \circ \ldots \circ \mu_{\theta_1}(T) = T_i$. However, since $T$ and $T_i$ only differ in finitely many arcs, and since $\gamma$ intersects infinitely many arcs in the triangulation $T_i$ this leads to a contradiction.
\end{rmq}

\begin{exmp}\label{E:directed}
 Consider the triangulations
 \[
  t(-\infty,\infty) = \{\pi_k \mid k \geq 0\} \cup \{\alpha_k \mid k \leq 0\} \; \; \text{and} \; \; t(\infty,\infty) = \{\pi_k \mid k \in \ZZ\}
 \]
 of $\IGC$. The $t(-\infty,\infty)$-admissible sequence $\underline{\theta}=(\alpha_{-i})_{i \geq 0}$ takes $t(-\infty,\infty)$ to $t(\infty,\infty)$, that is we have $\mu_{\underline{\theta}}(t(-\infty,\infty)) = t(\infty,\infty)$. However, by Remark \ref{R:no mutation} there is no $t(\infty,\infty)$-admissible sequence of arcs along which we can mutate to take $t(\infty,\infty)$ to $t(-\infty,\infty)$; the arc $\alpha_0 \in t(-\infty,\infty)$ for example intersects the infinitely many arcs $\pi_k \in t(\infty,\infty)$ with $k \leq -1$.
 \

\end{exmp}

\subsection{A preorder on triangulations of an infinitely marked surface}

Evidently, as we have seen in Example \ref{E:directed}, mutation along $T$-admissible sequences is ``directed'' in the sense that we might have a $T$-admissible sequence from a triangulation $T$ to a triangulation $T'$, but no way of mutating back from $T'$ to $T$ along a $T'$-admissible sequence. This naturally leads one to wonder if mutation along $T$-admissible sequences induces some sort of order on the set of triangulations of $(S,M)$. 

\begin{notation}
Let $T$ and $T'$ be triangulations of $(S,M)$. We write $T \leq_{s} T'$ if there is a $T$-admissible sequence $\underline{\theta}$ with $\mu_{\underline{\theta}}(T)=T'$.
\end{notation}

In this section we will show that $\leq_s$ induces a preorder on the set of triangulations of $(S,M)$.
The tricky part is showing transitivity. In the following, we introduce some notion and prove some results which will be very useful for this, and in fact will be used throughout the rest of this paper.

\begin{defi}
	Let $T$ be a triangulation of $(S,M)$ and let $\underline{\theta} = (\theta_i)_{i \in I}$
	be a $T$-admissible sequence. We say that $\underline{\theta}$ {\em leaves $\gamma \in T$ untouched} if $\mu_{\theta_l} \circ \ldots \circ \mu_{\theta_1}(\gamma) = \gamma$ for all $l \in I$.
\end{defi}

\begin{lem}\label{L:untouched equivalent}
 Let $T$ be a triangulation of $(S,M)$. A $T$-admissible sequence $\underline{\theta} = (\theta_i)_{i \in I}$ leaves $\gamma \in T$ untouched if and only if $\theta_i \neq \gamma$ for all $i \in I$.
\end{lem}

\begin{proof}
 Assume first that $\theta_j \neq \gamma$ for all $j \in I$. Then we have
 $\mu_{\theta_1}(\gamma) = \gamma$ and inductively assuming that $\mu_{\theta_i} \circ \ldots \circ \mu_{\theta_1}(\gamma) = \gamma$ for some $i \geq 1$, we obtain 
 \[
  \mu_{\theta_{i+1}} \circ \ldots \circ \mu_{\theta_1}(\gamma) = \mu^{\mu_{\theta_i} \circ \ldots \circ \mu_{\theta_1}(T)}_{\theta_{i+1}}(\gamma) = \gamma.
 \]
 
 To show the converse, assume that $\underline{\theta}$ does not leave $\gamma \in T$ untouched. Then there exists a $k \in I$ such that 
 \[
  \mu_{\theta_k} \circ \ldots \circ \mu_{\theta_1}(\gamma) \neq \gamma \; \; \text{and} \; \;  \mu_{\theta_i} \circ \ldots \circ \mu_{\theta_1}(\gamma) = \gamma \; \text{for all} \; i \leq k.
 \]
 It follows that 
 \[
  \mu_{\theta_k} \circ \ldots \circ \mu^T_{\theta_1}(\gamma) = \mu^{\mu_{\theta_{k-1}} \circ \ldots \circ \mu^T_{\theta_1}(\gamma)}_{\theta_k}(\gamma) \neq \gamma
 \]
  and therefore $\theta_k = \gamma$. This proves the claim.

\end{proof}

\begin{lem}\label{L:untouched}
  Let $T$ be a triangulation of $(S,M)$ and let $\underline{\theta}$ be a $T$-admissible sequence. Assume that $\delta \in T$ is mutable, and that $\underline{\theta}$ leaves all arcs in $S(\delta) \cup \{\delta\}$ untouched. Then $\underline{\theta}$ is a $\mu_{\delta}(T)$-admissible sequence with $\mu_{\underline{\theta}}(\mu_{\delta}(T)) = \mu_{\delta}(\mu_{\underline{\theta}(T)})$.
\end{lem}

\begin{proof}
Let $\underline{\theta} = (\theta_i)_{i \in I}$. Inductively applying Lemma \ref{L:commutativity} implies that for all $k \in I$ the finite sequence $(\delta, \theta_1, \ldots, \theta_k)$ is $T$-admissible and for all $\gamma \in T$ we have
 \[
   \mu_{\delta} \circ \mu_{\theta_k} \circ \ldots \circ \mu^T_{\theta_1}(\gamma) = \mu_{\theta_k} \circ \ldots \circ \mu_{\theta_1} \circ \mu^T_{\delta}(\gamma).
 \] 
 Therefore, for all $k \in I$ the sequence $\underline{\theta}_k = (\theta_1, \ldots, \theta_k)$ is $\mu_{\delta}(T)$-admissible. Assume that $\delta' = \mu^T_{\delta}(\delta)$. The sequence $\underline{\theta}_k$ leaves $\delta' \in \mu_{\delta}(T)$ untouched: Since $\delta \in T$ and $\delta$ and $\delta'$ intersect, we have $\delta' \neq \theta_1 \in T$. Furthermore, since for all $k \geq 1$ the sequence $\underline{\theta}_k$ leaves $\delta$ untouched, we have $\delta \in \mu^T_{\underline{\theta}_k}(T)$ and therefore $\delta' \neq \theta_{k+1} \in \mu^T_{\underline{\theta}_k}(T)$. It follows from Lemma \ref{L:untouched equivalent} that $\underline{\theta}_k$ as a $\mu_{\delta}(T)$-admissible sequence leaves $\delta'$ untouched. Therefore for all $k \in I$ we have 
 \begin{eqnarray}\label{delta}
  \mu_{\theta_k} \circ \ldots \circ \mu^{\mu_{\delta}(T)}_{\theta_1}(\delta') = \delta'.
 \end{eqnarray}
 
 Consider now $\gamma \in \mu_{\delta}(T)$ with $\gamma \neq \delta'$. Then we have $\gamma \in T \setminus \{\delta\}$. Since $\underline{\theta}$ is $T$-admissible, there exists an $l \in I$ such that for all $k \geq l$ we have
 \[
  \mu_{\theta_k} \circ \ldots \circ \mu^T_{\theta_1}(\gamma) = \mu_{\theta_l} \circ \ldots \circ \mu^T_{\theta_1}(\gamma).
 \]
 Because we have $\gamma \in T \setminus \{\delta\}$ and since $\underline{\theta}$ leaves $\delta \in T$ untouched, for all $k \in I$ we  have $\mu_{\underline{\theta}_k}(\gamma) \in \mu_{\underline{\theta}_k}^T(T) \setminus \{\delta\}$.  It follows that for all $k \geq l$ we have
 \begin{align*}
  \mu_{\theta_k} \circ \ldots \circ \mu^{\mu_\delta(T)}_{\theta_1}(\gamma) 	&= \mu_{\theta_k} \circ \ldots \circ \mu_{\theta_1}(\mu^T_\delta(\gamma))
										= \mu_{\delta} \circ \mu_{\theta_k} \circ \ldots \circ \mu^T_{\theta_1}(\gamma)\\
										&= \mu_{\theta_k} \circ \ldots \circ \mu^T_{\theta_1}(\gamma) 
										= \mu_{\theta_l} \circ \ldots \circ \mu^T_{\theta_1}(\gamma) = \mu^T_{\underline{\theta}}(\gamma). 
 \end{align*}
Therefore $\underline{\theta}$ is $\mu_{\delta}(T)$-admissible with
\[
 \mu^{\mu_{\delta}(T)}_{\underline{\theta}}(\gamma) =  \begin{cases}
								    \delta' \; \text{if} \; \gamma = \delta'\\
 								    \mu^T_{\underline{\theta}}(\gamma) \; \text{otherwise},
 								  \end{cases}
\]
and we have
\[
 \mu_{\underline{\theta}}(\mu_{\delta}(T)) = (\mu_{\underline{\theta}}(T) \setminus \{\delta\}) \cup \{\delta'\} = \mu_{\delta}(\mu_{\underline{\theta}}(T)).
\]
\end{proof}

\begin{lem}\label{L:precomposing}
 Let $T$ be a triangulation of $(S,M)$ and let $\underline{\theta} = (\theta_i)_{i \in I}$ be a $T$-admissible sequence. If $\delta$ is a mutable arc in $\mu_{\underline{\theta}}(T)$ then there exists an $r \in I$ such that for all $l \geq r$ the sequence
  \[
  \underline{\theta} \cup_l (\delta) = (\theta_1, \ldots, \theta_l, \delta, \theta_{l+1}, \theta_{l+2}, \ldots)
  \]
  is a $T$-admissible sequence with $\mu_{\underline{\theta}\cup_l (\delta)}(\gamma) = \mu_{\delta} (\mu_{\underline{\theta}}(\gamma))$ for all $\gamma \in T$.
\end{lem}

\begin{proof} 
Set $T'=\mu_{\underline{\theta}}(T)$. 
Since $\delta \in T'$ is mutable, we can consider the quadrilateral $S(\delta)$ in $T'$ with diagonal $\delta$. Furthermore, because $\underline{\theta}$ is $T$-admissible with $\mu_{\underline{\theta}}(T)=T'$, there exists an $r \in I$ such that
\begin{eqnarray} \label{E:stagnant1}
  \{\delta\}\cup S(\delta) \subseteq \mu_{\theta_{r}} \circ \ldots \circ \mu_{\theta_1}(T) = T_r
\end{eqnarray}
and such that 
\begin{eqnarray}\label{E:stagnant2}
  \mu_{\theta_k} \circ \ldots \circ \mu^{T_r}_{\theta_{r+1}}(\gamma) = \gamma
\end{eqnarray}
for all $\gamma \in S(\delta) \cup \{\delta\}$ and $r \leq k \in I$. Fix now an $l \in I$ with $l \geq r$ and set
\[
\underline{\theta} \cup_l (\delta) = (\theta_1, \ldots, \theta_l, \delta, \theta_{l+1}, \theta_{l+2}, \ldots).
\]
This sequence is $T$-admissible: It is clear that $\theta_1$ is mutable in $T$ and that $\theta_i \in \mu_{\theta_{i-1}} \circ \ldots \circ \mu_{\theta_1}(T)$ is mutable for $2 \leq i \leq l$. Furthermore, by (\ref{E:stagnant1}), we have that $\delta$ is mutable in $\mu_{\theta_{l}} \circ \ldots \circ \mu_{\theta_1}(\tilde{T})$. Finally, setting $T_l = \mu_{\theta_l} \circ \ldots \circ \mu_{\theta_1}(T)$, it follows from (\ref{E:stagnant2}) that the $T_l$-admissible sequence 
\[
 \underline{\theta}_{l+1}=(\theta_i)_{i \in I \setminus \{1, \ldots, l\}}
\]
leaves all arcs in $\{\delta\} \cup S(\delta)$ untouched. By Lemma \ref{L:untouched} we obtain that $\underline{\theta}_{l+1}$ is a $\mu_{\delta}(T_l)$-admissible sequence, and therefore the sequence $\underline{\theta} \cup_l (\delta)$ is $T$-admissible. Furthermore, again by Lemma \ref{L:untouched}, for all $\gamma \in T$ we obtain
\begin{eqnarray*}
 \mu_{\underline{\theta}\cup_l (\delta)}(\gamma) 	&=& \mu_{\underline{\theta}_{l+1}} (\mu^{T_l}_{\delta}(\mu_{\theta_l} \circ \ldots \circ \mu_{\theta_1}(\gamma)))\\
							&=& \mu_{\delta}(\mu_{\underline{\theta}_{l+1}}^{T_l} (\mu_{\theta_l} \circ \ldots \circ \mu_{\theta_1}(\gamma)))
							= \mu_{\delta}(\mu_{\underline{\theta}}(\gamma)),
\end{eqnarray*}
which proves the claim.

\end{proof}

\begin{rmq}\label{R:precomposing}
 With the notation as in Lemma \ref{L:precomposing}, assume that $S(\delta)$ is the quadrilateral in $\mu_{\underline{\theta}}(T)$ with diagonal $\delta$. In the proof of Lemma \ref{L:precomposing}, we picked $r \in I$ big enough so that not only the the desired property is satisfied but so that we furthermore have $\{\delta\} \cup S(\delta) \subseteq \mu_{\theta_r} \circ \ldots \circ \mu_{\theta_1}(T)$ and $(\theta_i)_{i >r}$ leaves every arc in $\{\delta\} \cup S(\delta)$ untouched. We will use this aspect of the construction in the proof of Proposition \ref{P:transitivity}.
\end{rmq}

\begin{prop}\label{P:transitivity}
Let $T$, $T'$ and $T''$ be triangulations of $(S,M)$. Assume there exists a $T$-admissible sequence $\underline{\alpha}$ such that $\mu_{\underline{\alpha}}(T) = T'$ and a $T'$-admissible sequence $\underline{\beta}$ such that $\mu_{\underline{\beta}}(T') = T''$. Then there exists a $T$-admissible sequence $\underline{\gamma}$ such that $\mu_{\underline{\gamma}}(T)=T''$.
\end{prop}

\begin{proof}
If $\underline{\alpha}$ is a finite sequence, i.e.\ $\underline{\alpha} = (\alpha_1, \ldots, \alpha_n)$ for some $n \in \ZZ_{>0}$ the statement is trivial -- we can just set $\underline{\gamma} = (\alpha_1, \ldots, \alpha_n, \underline{\beta})$. Furthermore, if $\underline{\beta}$ is a finite sequence then the statement follows by iteratively applying Lemma \ref{L:precomposing}. Assume thus that $\underline{\alpha} = (\alpha_i)_{i \in \ZZ_{>0}}$ and $\underline{\beta} = (\beta_i)_{i \in \ZZ_{>0}}$. We build a $T$-admissible sequence $\underline{\gamma}$ with $\mu_{\underline{\gamma}}(T)=T''$ by interlacing the sequences $\underline{\alpha}$ and $\underline{\beta}$ in the following way:

Since $\beta_1 \in T'$ is mutable, by Lemma \ref{L:precomposing} there exists an $l_1 \in I_\alpha$ such that
\[
 \underline{\alpha} \cup_{l_1} (\beta_1) = (\alpha_1, \ldots, \alpha_{l_1}, \beta_1, \alpha_{l_1+1}, \ldots)
\]
is $T$-admissible with $\mu_{\underline{\alpha} \cup_{l_1} (\beta_1)}(\gamma) = \mu_{\beta_1} (\mu_{\underline{\alpha}}(\gamma))$ for all $\gamma \in T$. By Remark \ref{R:precomposing} we may assume that $l_1$ is big enough such that, if $S(\beta_1)$ is the quadrilateral in $\mu_{\underline{\alpha}}(T)$ with diagonal $\beta_1$, we have $\{\beta_1\} \cup S(\beta_1) \subseteq \mu_{\alpha_{l_1}} \circ \ldots \circ \mu_{\alpha_1}(T)$ and the $\mu_{\alpha_{l_1}} \circ \ldots \circ \mu_{\alpha_1}(T)$-admissible sequence $(\alpha_j)_{j > l_1}$ leaves all arcs in $\{\beta_1\} \cup S(\beta_1)$ untouched. By iteratively applying Lemma \ref{L:precomposing} for all $i \geq 2$ we can pick $l_{i} \in \ZZ_{>0}$ with $l_i > l_{i-1}$ and set $l_0 = 0$, such that
\begin{eqnarray*}
 \underline{\alpha} \cup (\beta_1, \ldots, \beta_{i}) &=& \underline{\alpha} \cup (\beta_1, \ldots, \beta_{i-1}) \cup_{l_i} (\beta_{i}) \\
							  &=& (\alpha_1, \ldots, \alpha_{l_1}, \beta_1, \alpha_{l_1+1}, \ldots, \alpha_{l_{i}}, \beta_{i}, \alpha_{l_{i}+1}, \ldots)\\
							  &=& ((\alpha_{l_{k-1}+1}, \ldots, \alpha_{l_k}, \beta_{k})_{1 \leq k \leq i}, (\alpha_j)_{j \geq l_i + 1})
\end{eqnarray*}
is $T$-admissible with 
\[
 \mu_{\underline{\alpha} \cup (\beta_1, \ldots, \beta_i)}(\gamma) = \mu_{\beta_i}\circ \ldots \circ \mu_{\beta_1}(\mu_{\underline{\alpha}}(\gamma)).
\]
For $i \geq 2$ assume that $S(\beta_i)$ is the quadrilateral in  $\mu_{\underline{\alpha} \cup (\beta_1, \ldots, \beta_{i-1})}(T)$ with diagonal $\beta_i$. By Remark \ref{R:precomposing} we can assume without loss of generality that for each $i \in \ZZ_{>0}$ we picked $l_i \in \ZZ_{>0}$ big enough such that $(\alpha_j)_{j > l_i}$ leaves all arcs in $S(\beta_i) \cup \{\beta_i\}$ untouched. Set
\[
 \underline{\gamma} = (\gamma_i)_{i \in \ZZ_{>0}} = ((\alpha_{l_{i-1}+1}, \ldots, \alpha_{l_i}))_{i \in I_\beta}.
\]
In the following we prove that this is the desired $T$-admissible sequence with $\mu_{\underline{\gamma}}(T) = T''$.

Notice that if we consider finite length sequences of the form $(\gamma_1, \ldots, \gamma_k)$ for $k \geq 1$ then as sets we have $\{\gamma_1, \ldots, \gamma_k\} = \{\alpha_1, \ldots, \alpha_m, \beta_1, \ldots, \beta_n\}$ for some $m,n \in \ZZ_{>0}$. Iteratively applying Lemma \ref{L:commutativity}, and using the fact that $(\alpha_j)_{j > l_i}$ leaves all arcs in $\{\beta_i\} \cup S(\beta_i)$ untouched, we can push the $\beta_i$ towards the end of the sequence and obtain a $T$-admissible sequence
\[
 (\alpha_1, \ldots, \alpha_m, \beta_1, \ldots, \beta_n)
\]
and for all $\delta \in T$ we have
\[
 \mu_{\beta_n} \circ \ldots \mu_{\beta_1} \circ \mu_{\alpha_m} \circ \ldots \circ \mu_{\alpha_1}(\delta) = \mu_{\gamma_k} \circ \ldots \circ \mu_{\gamma_1}(\delta).
\]

We now show that $\underline{\gamma}$ is a $T$-admissible sequence. Clearly $\gamma_1 = \alpha_1$ is mutable in $T$. For $i \geq 2$, there exists a $j \in \ZZ_{>0}$ such that the the first $i$ entries of the sequence $\underline{\alpha} \cup \{\beta_1, \ldots, \beta_j\}$ coincide with the sequence $(\gamma_1, \ldots, \gamma_i)$. Since $\underline{\alpha} \cup \{\beta_1, \ldots, \beta_j\}$ is $T$-admissible, it follows that $\gamma_i \in \mu_{\gamma_{i-1}} \circ \ldots \circ \mu_{\gamma_1}(T)$ is mutable. To show that the sequence is $T$-admissible, it thus remains to show that for each $\delta \in T$ there exists an $l >0$ such that for all $k \geq l$ we have
\[
 \mu_{\gamma_k} \circ \ldots \circ \mu_{\gamma_1}(\delta) = \mu_{\gamma_l} \circ \ldots \circ \mu_{\gamma_1}(\delta).
\]
Let $\mu_{\underline{\alpha}}(\delta) = \delta' \in T'$. Assume first that $\delta'$ is not mutable in $T'$. Then it is not mutable in $\mu_{\beta_l} \circ \ldots \circ \mu_{\beta_1}(T')$ for any $l \geq 1$ by Lemma \ref{L:mutable1}. It follows that $\beta_i \neq \delta'$ for all $i \geq 1$.
Since $\underline{\alpha}$ is $T$-admissible there exists an $l \in \ZZ_{>0}$ such that for all $k \geq l$ we have
\[
\mu_{\alpha_k} \circ \ldots \circ \mu_{\alpha_1}(\delta) = \mu_{\underline{\alpha}}(\delta) = \delta'.
\]
Pick $m \in \ZZ_{>0}$ such that $\{\gamma_1, \ldots , \gamma_m\} = \{\alpha_1, \ldots, \alpha_l, \beta_1, \ldots, \beta_p\}$ for some $p \in \ZZ_{>0}$. For all $k \geq m$ we have $\{\gamma_1, \ldots, \gamma_k\} =  \{\alpha_1, \ldots, \alpha_s, \beta_1, \ldots, \beta_t\}$ for some $s \geq l$ and $t \geq p$ and we obtain
\begin{align}
 \mu_{\gamma_k}\circ \ldots \circ \mu_{\gamma_1}(\delta) &= \mu_{\beta_t} \circ \ldots \circ \mu_{\beta_1} \circ \mu_{\alpha_s} \circ \ldots \circ \mu_{\alpha_1}(\delta) \label{E:correct1}\\
 											&= \mu_{\beta_t} \circ \ldots \circ \mu_{\beta_1} (\delta')
											= \delta', \nonumber
\end{align}
where the last equality holds since $\delta' \notin \{\beta_1, \ldots, \beta_t\}$.
This proves the claim in this case.

On the other hand, if $\delta'$ is mutable in $T'$ then we can consider the quadrilateral $S(\delta')$ in $T'$ with diagonal $\delta'$. There exists an $l \in \ZZ_{>0}$ such that 
\[
 \{\delta'\} \cup S(\delta') \subseteq \mu_{\alpha_l} \circ \ldots \circ \mu_{\alpha_1}(T) = T_l
\]
and for all $k \geq l$
\[
 \mu_{\alpha_k} \circ \ldots \circ \mu^{T_l}_{\alpha_{l+1}}(x) = x
\]
for all $x \in \{\delta'\} \cup S(\delta')$.  If $\delta'' \neq \delta'$ is the other diagonal of $S(\delta')$ then it follows from the definition of mutation that $\mu_{\delta'}^{\mu_{\alpha_k} \circ \ldots \circ \mu_{\alpha_1}(T)}(\delta') = \delta''$ for all $k \geq l$. Since $\underline{\beta}$ is $T'$-admissible there exists a $r \in \ZZ_{>0}$ such that $\mu_{\beta_k} \circ \ldots \circ \mu_{\beta_1}(\delta') = \delta''$ for all $k \geq r$. Pick $m \in \ZZ_{>0}$ such that $\{\gamma_1, \ldots , \gamma_m\} = \{\alpha_1, \ldots, \alpha_q, \beta_1, \ldots, \beta_p\}$ for some $p \geq r$ and $q \geq l$. Let $k \geq m$ with $\{\gamma_1, \ldots, \gamma_k\} =  \{\alpha_1, \ldots, \alpha_s, \beta_1, \ldots, \beta_t\}$ for some $s \geq l$ and $t \geq p$. We obtain
\begin{align}
 \mu_{\gamma_k}\circ \ldots \circ \mu_{\gamma_1}(\delta) &= \mu_{\beta_t} \circ \ldots \circ \mu_{\beta_1} \circ \mu_{\alpha_s} \circ \ldots \circ \mu_{\alpha_1}(\delta) \label{E:correct2}\\
 											&= \mu_{\beta_t} \circ \ldots \circ \mu^{\mu_{\alpha_s} \circ \ldots \circ \mu_{\alpha_1}(\delta)}_{\beta_1} (\delta')
											= \delta''. \nonumber
\end{align}
This proves that the sequence $\underline{\gamma}$ is $T$-admissible. Furthermore, (\ref{E:correct1}) and (\ref{E:correct2}) ensure that 
\[
\mu_{\underline{\gamma}}(\delta) = \mu_{\underline{\beta}} (\mu_{\underline{\alpha}}(\delta)) \; \; \text{for all $\delta \in T$}.
\]

\end{proof}

\begin{theorem}\label{T:partial order}
 The relation $\leq_{s}$ defines a preorder on the set of triangulations of $(S,M)$.
\end{theorem}

\begin{proof}
 Reflexivity is clear and transitivity follows from Proposition \ref{P:transitivity}.
\end{proof}
	
\section{Strong mutation equivalence}

The preorder $\leq_s$ induces an equivalence relation on the set of triangulations of a fixed infinitely marked surface.

\begin{defi}
Let $T$ and $T'$ be triangulations of $(S,M)$. We say that $T$ and $T'$ are {\em strongly mutation equivalent} if $T \leq_s T'$ and $T' \leq_s T$.
\end{defi}

This section is dedicated to understanding when two triangulations of $\IG$, respectively of $\IGC$, are strongly mutation equivalent. Before we provide a complete classification of strong mutation equivalence classes in both cases, we introduce useful notation and make some observations.

\begin{defi}
	Let $T$ be a triangulation of $\IG$ (respectively of $\IGC$) and set $\E = \E(\IG)$ (respectively $\E = \E(\IGC)$). A {\em finite subpolygon of $T$} is a finite set of vertices $P = \{x_1, \ldots, x_k\} \subseteq \mathbb{Z} \cup \{\pm \infty\}$ with $k \geq 3$ that can be ordered such that $x_1 < x_2 < \ldots < x_k$ and with $(x_1,x_k) \in T \cup \E$ and for all $1 \leq i < k$ we have $(x_i,x_{i+1}) \in T \cup \E$.
	
	If $P$ is a finite subpolygon of $T$ as above, we denote by $\s(P)$ the set $\s(P) = \{(x_i,x_{j}) \mid 1 \leq i < j \leq k\}$. We call 
	\[ 
	\E(P) =\{(x_i,x_{i+1}) \mid 1 \leq i < k-1\} \cup \{(x_1,x_k)\} 
	\]
        the {\em edges of $P$} and 
        \[
         \A(P) = \s(P) \setminus \E(P)
        \]
       the {\em arcs of $P$}. 
\end{defi}

\begin{notation}
Let $T$ be a triangulation of $\IG$ (respectively of $\IGC$) and let $P \subseteq \mathbb{Z} \cup \{\pm \infty\}$ be a set of vertices. Then we denote by $T \mid_P$ the set of arcs
\[
  T \mid_P = \{(a,b) \in T \mid a,b \in P\}.
\]
\end{notation}

\begin{rmq}\label{R:subtriangulation}
Locally, triangulations of $\IG$ and $\IGC$ behave like triangulations of finite polygons: If $P$ is a finite subpolygon of $T$, then  $T \mid_P$ is a triangulation of the polygon with vertices $P$,
i.e.\ a maximal set of non intersecting arcs with endpoints in $P$, and we call it a {\em finite subtriangulation of $T$}.
\end{rmq}

\begin{rmq}\label{R:fin. subtr.}
Assume $T$ and $T'$ are both triangulations of $\IG$, respectively of $\IGC$, with finite subtriangulation $T \mid_P$ and $T' \mid_{P'}$ 
for some finite subpolygons $P$ of $T$ and $P'$ of $T'$ such that $P' \subseteq P$. Then -- via mutations in the finite subpolygon with vertices $P$ -- there exists a finite $T$-admissible sequence $\underline{\theta}$ such that $\mu_{\underline{\theta}}(T)\mid_{P'} = T'\mid_{P'}$ and $\underline{\theta}$ leaves all arcs in $T \setminus \A(P)$ untouched.
\end{rmq} 
	
The following results will be useful to describe strong mutation equivalence classes.

\begin{lem}\label{L:union of subpolygons}
 Let $T$ be a triangulation of $\IG$, respectively of $\IGC$ with finite subpolygons $P_i$ for $i \in \ZZ_{>0}$ such that $\A(P_i) \cap \s(P_j) = \varnothing$ for $i \neq j$. Let $T'$ be a triangulation of $\IG$, respectively of $\IGC$, with finite subpolygons $P'_i$ for $i \in \ZZ_{>0}$ such that $P'_i \subseteq P_i$ for all $i \in \ZZ_{>0}$.
 Then there exists a $T$-admissible sequence $\underline{\theta}$ such that
 \[
  \mu_{\underline{\theta}}(T) \mid_{\bigcup_{i \in \ZZ_{>0}}P'_i} = T' \mid_{\bigcup_{i \in \ZZ_{>0}}P'_i}
 \]
 and such that $\underline{\theta}$ leaves all arcs in $T \setminus \bigcup_{i \in \ZZ_{>0}}\A(P_i)$ untouched.

\end{lem}

\begin{proof}
 By Remark \ref{R:fin. subtr.} for all $i \in \ZZ_{>0}$ there exists a finite length $T$-admissible sequence $\underline{\theta}^i$ with $\mu_{\underline{\theta^i}}(T) \mid_{P'_i} = T' \mid_{P'_i}$ and such that $\underline{\theta}^i$ leaves all arcs in $T \setminus \A(P_i)$ untouched. We label the arcs in the sequence $\underline{\theta}^i$ by $\underline{\theta}^i = (\theta_j)_{l_{i-1} < j \leq l_i}$ where we set $l_0 = 0$ and for $i \geq 1$ we pick $l_i \in \ZZ_{>0}$ such that $(l_i-l_{i-1})$ is the length of the admissible sequence $\underline{\theta}^i$.
 We make the following observation
 \begin{itemize}
  \item[(i)] Let $\gamma \in T$. Then we have $\gamma \in \A(P_i)$ if and only if $\mu_{\theta_k} \circ \ldots \circ \mu_{\theta_{l_{i-1}+1}}(\gamma) \in \A(P_i)$ for all $l_{i-1}< k \leq l_i$. 
 \end{itemize}
 It is clear that if $\gamma \notin \A(P_i)$, then, since $\underline{\theta}^i$ leaves $\gamma$ untouched, we also have $\mu_{\theta_k} \circ \ldots \circ \mu_{\theta_{l_{i-1}+1}}(\gamma) = \gamma \notin \A(P_i)$ for all $l_{i-1}< k \leq l_i$. On the other hand, assume that $\gamma \in T \cap \A(P_i)$. Set $T_0=T$ and for $1 \leq m \leq k$ set $T_k = \mu_{\theta_k} \circ \ldots \circ \mu_{\theta_1}(T)$. We show the claim by induction. For $i \geq 0$ assume that $\gamma \in T_i \cap \A(P_i)$. Since $\underline{\theta}^i$ leaves all arcs in $\E(P_i)$ untouched, we obtain that $P_i$ is a finite subpolygon of $T_i$. Consider the quadrilateral $S(\gamma)$ in $T_i$ with diagonal $\gamma$. We have $S(\gamma) \subseteq \s(P_i)$ and therefore the other diagonal $\gamma' \neq \gamma$ in $S(\gamma)$ also lies in $\A(P_i)$. It follows that $\mu_{\theta_{m+1}}(\gamma) \in \{\gamma, \gamma'\}$ lies in $\A(P_i)$.
 
 Set $\underline{\theta} = (\theta_i)_{i \geq 1}$. Clearly the sequence $(\theta_1)$ of length one is $T$-admissible and for all $\gamma \in T$ we have
 \[
  \mu_{\theta_1}(\gamma) = \begin{cases} \mu_{\theta_1}(\gamma) \; \text{if $\gamma \in \A(P_1)$}\\
                                                              \gamma \; \text{otherwise}.
                                                             \end{cases}
 \]
 
 We show that for all $m \geq 1$ the sequence $(\theta_i)_{1 \leq i \leq m}$ is $T$-admissible and, setting $j \geq 1$ such that $l_{j-1} < m \leq l_j$, for all $\gamma \in T$ we have
 \begin{eqnarray}\label{E:condition}
  \mu_{\theta_m} \circ \ldots \circ \mu_{\theta_1}(\gamma) = \begin{cases}
                                                              \mu_{\theta_m} \circ \ldots \circ \mu_{\theta_{l_{j-1}+1}}(\gamma) \; \text{if $\gamma \in \A(P_j)$}\\
                                                              \mu_{\theta_{l_i}} \circ \ldots \circ \mu_{\theta_{l_{i-1}+1}}(\gamma) \; \text{if $\gamma \in \A(P_i)$ for $1 \leq i < j$}\\
                                                              \gamma \; \text{otherwise}.
                                                             \end{cases}
 \end{eqnarray}

 Assume this condition holds for $m \geq 1$, and let $j \geq 1$ be such that $l_{j-1} < m \leq l_j$. We show that it also holds for $m+1$. Consider thus the sequence $(\theta_i)_{1 \leq i \leq m+1}$. We distinguish two cases.
 
 Case 1: Assume that $l_{j-1} < m < m+1 \leq l_j$. Then $\theta_{m+1}$ is mutable in $\mu_{\theta_m} \circ \ldots \circ \mu_{\theta_{l_{j-1}+1}}(T)$. Consider the quadrilateral $S(\theta_{m+1})$ in $\mu_{\theta_m} \circ \ldots \circ \mu_{\theta_{l_{j-1}+1}}(T)$ with diagonal $\theta_{m+1}$. We show that in fact we have $\{\theta_{m+1}\} \cup S(\theta_{m+1}) \subseteq \mu_{\theta_m} \circ \ldots \circ \mu_{\theta_1}(T)$. 
 
 Assume thus that $\alpha \in \{\theta_{m+1}\} \cup S(\theta_{m+1})$. There exists a $\beta \in T$ with $\alpha = \mu_{\theta_m} \circ \ldots \circ \mu_{\theta_{l_{j-1}+1}}(\beta)$. If $\alpha \in \A(P_i)$ then by (i) we have $\beta \in \A(P_i)$ and therefore 
 \[
  \alpha = \mu_{\theta_m} \circ \ldots \circ \mu_{\theta_{l_{j-1}+1}}(\beta) = \mu_{\theta_m} \circ \ldots \circ \mu_{\theta_1}(\beta) \in \mu_{\theta_m} \circ \ldots \circ \mu_{\theta_1}(T).
 \]
 Assume on the other hand that $\alpha \in \E(P_i)$. Then, since $P_i$ is a subpolygon of $T$, we have $\alpha \in T$. Furthermore, we have $\E(P_i) \cap \A(P_j) = \varnothing$ for all $j \geq 1$: this is clear for $j = i$ and follows from the assumption $\s(P_i) \cap \A(P_j) = \varnothing$ for $i \neq j$. It follows that $\alpha \notin \{\theta_i \mid i \geq 1\}$ and therefore the sequence $(\theta_i)_{1 \leq i \leq m}$ leaves $\alpha$ untouched and we have $\alpha \in \mu_{\theta_m} \circ \ldots \circ \mu_{\theta_1}(T)$. It follows that $\{\theta_{m+1}\} \cup S(\theta_{m+1}) \subseteq \mu_{\theta_m} \circ \ldots \circ \mu_{\theta_1}(T)$ and therefore the sequence $(\theta_i)_{1 \leq i \leq m+1}$ is $T$-admissible. Furthermore, since $\theta_{m+1} \in \A(P_j)$ it leaves all arcs that are not in $\A(P_j)$ untouched. By (i) and since $\A(P_i) \cap \A(P_j) = \varnothing$ for $i \neq j$, we have $\mu_{\theta_{l_i}} \circ \ldots \circ \mu_{\theta_{l_{i-1}+1}}(\gamma) \notin \A(P_j)$ if $\gamma \in \A(P_i)$ with $i \neq j$. It follows that
 \[
  \mu_{\theta_{m+1}} \circ \ldots \circ \mu_{\theta_1}(\gamma) = \begin{cases}
                                                              \mu_{\theta_{m+1}} \circ \ldots \circ \mu_{\theta_{l_{j-1}+1}}(\gamma) \; \text{if $\gamma \in \A(P_j)$}\\
                                                              \mu_{\theta_{l_i}} \circ \ldots \circ \mu_{\theta_{l_{i-1}+1}}(\gamma) \; \text{if $\gamma \in \A(P_i)$ for $1 \leq i < j$}\\
                                                              \gamma \; \text{otherwise}.
                                                             \end{cases}
 \]
 
 Case 2: Assume that $m +1 = l_j +1$. Then $\theta_{m+1} \in \A(P_{j+1})$ is mutable in $T$. Consider the quadrilateral $S(\theta_{m+1})$ in $T$ with diagonal $\theta_{m+1}$. We have $S(\theta_{m+1}) \subseteq \s(P_{j+1})$. Since $\s(P_{j+1}) \cap \A(P_i) = \varnothing$ for all $1 \leq i \leq j$, the sequence $(\theta_i)_{1 \leq i \leq m}$ leaves $\theta_{m+1} \in T$ untouched. By iteratively applying Lemma \ref{L:commutativity} we obtain that $(\theta_i)_{1 \leq i \leq m+1}$ is $T$-admissible with
 \[
  \mu_{\theta_{m+1}} \circ \ldots \circ \mu_{\theta_1}(\gamma) = \begin{cases}
                                                              \mu_{\theta_{m+1}}(\gamma) \; \text{if $\gamma \in \A(P_{j+1})$}\\
                                                              \mu_{\theta_{l_i}} \circ \ldots \circ \mu_{\theta_{l_{i-1}+1}}(\gamma) \; \text{if $\gamma \in \A(P_i)$ for $1 \leq i \leq j$}\\
                                                              \gamma \; \text{otherwise}.
                                                             \end{cases}
 \]
 
 Therefore, for every $m \in \ZZ_{>0}$ the sequence $(\theta_i)_{1 \leq i \leq m}$ is $T$-admissible and satisfies condition (\ref{E:condition}). Consider now the sequence $\underline{\theta}=(\theta_i)_{i \in \ZZ_{>0}}$. Pick $\gamma \in T$. Then, if $\gamma \in \A(P_i)$ for some $i \in \ZZ_{>0}$ for all $k \geq l_i$ we have
 \[
  \mu_{\theta_k} \circ \ldots \circ \mu_{\theta_1}(\gamma) = \mu_{\theta_{l_i}} \circ \ldots \circ \mu_{\theta_1}(\gamma) = \mu_{\theta_{l_i}} \circ \ldots \circ \mu_{\theta_{l_{i-1}+1}}(\gamma)
 \]
 and if $\gamma \notin \bigcup_{i \in \ZZ_{>)}}\A(P_i)$, for all $k \geq 1$ we have $\mu_{\theta_k} \circ \ldots \circ \mu_{\theta_1}(\gamma) = \gamma$. It follows that $\underline{\theta}$ is $T$-admissible with
 \[
  \mu_{\underline{\theta}}(T) \mid_{\bigcup_{i \in \ZZ_{>0}}P'_i} = T' \mid_{\bigcup_{i \in \ZZ_{>0}}P'_i}
 \]
 and that $\underline{\theta}$ leaves all arcs in $T \setminus \bigcup_{i \in \ZZ_{>0}}\A(P_i)$ untouched.

\end{proof}

	\begin{lem} \label{L:mutation in fountain sections}
	Let $T$ and $T'$ be triangulations of $\IG$, respectively of $\IGC$.
	\begin{enumerate}
	
	  \item If $T$ and $T'$ both have a right fountain at $b \in \ZZ$, then there exists a $T$-admissible sequence $\underline{\theta}$ with 
	  \[\mu_{\underline{\theta}}(T) \mid_{[b,\infty)} = T' \mid_{[b,\infty)}\] and such that $\underline{\theta}$ leaves all arcs in $T \setminus T\mid_{[b,\infty)}$ untouched.
	  
	  \item If $T$ and $T'$ both have a left fountain at $a \in \ZZ$, then there exists a $T$-admissible sequence $\underline{\theta}$ with \[\mu_{\underline{\theta}}(T) \mid_{(-\infty,a]} = T' \mid_{(-\infty,a]}\] and 
	  such that $\underline{\theta}$ leaves all arcs in $T \setminus T\mid_{(-\infty,a]}$ untouched.
	  
	\item If $T$ and $T'$ are locally finite, then there exists a $T$-admissible sequence $\underline{\theta}$ with \[\mu_{\underline{\theta}}(T) = T'.\]
	  
	  \item If $T$ and $T'$ are triangulations of $\IGC$ and both have a right fountain at $b = \infty$ (respectively at $b = -\infty$), then there exists a $T$-admissible sequence $\underline{\theta}$ and a $k \in \ZZ$ such that $\pi_k \in T'$ (respectively $\alpha_k \in T'$) with 
	  \[\mu_{\underline{\theta}}(T) \mid_{\{b\} \cup [k,\infty)} = T' \mid_{\{b\} \cup [k,\infty)}\] and such that $\underline{\theta}$ leaves all arcs in $T \setminus T\mid_{\{b\} \cup [k,\infty)}$ untouched.
	  
	  \item If $T$ and $T'$ are triangulations of $\IGC$ and both have a left fountain at $a = \infty$ (respectively at $a = -\infty$), then there exists a $T$-admissible sequence $\underline{\theta}$ and a $k \in \ZZ$ such that $\pi_k \in T'$ (respectively $\alpha_k \in T'$) with  
	  \[\mu_{\underline{\theta}}(T) \mid_{\{a\} \cup (-\infty, k]} = T' \mid_{\{a\} \cup (-\infty, k]}\] and such that $\underline{\theta}$ leaves all arcs in $T \setminus T\mid_{\{a\} \cup (-\infty, k]}$ untouched.

	\end{enumerate}

	\end{lem}

	\begin{proof}
	
	We start by showing (1):
	Let $T$ be any triangulation with a right fountain at $b \in \ZZ$ and consider the strictly increasing sequence $(k_i)_{i \geq 1}$, where 
	\[
	\{k_i \mid i \geq 1\} = \{m \mid (b,m) \in T\}. 
	\]
 Let further $T'$ be a triangulation such that $T' \mid_{[b,\infty[} = \{(b,k) \mid k \geq b+2\}$. Then, setting $b = k_0$, the sets $P_i = \{b\} \cup [k_{i-1},k_i]$ are finite subpolygons of both $T$ and $T'$ and $\A(P_i) \cap \s(P_j) = \varnothing$ if $i \neq j$. Furthermore, we have $\bigcup_{i \in \ZZ_{>0}}P_i = [b,\infty)$. It follows from Lemma \ref{L:union of subpolygons} that there exists both a $T$-admissible sequence $\underline{\theta}$ with $\mu_{\underline{\theta}}(T)\mid_{[b,\infty)} = T' \mid_{[b,\infty)}$ that leaves all arcs in $T \setminus T\mid_{[b,\infty)} \subseteq T \setminus \bigcup_{i \in \ZZ_{>)}}\A(P_i)$ untouched and symmetrically there exists a $T'$-admissible sequence $\underline{\theta}'$ with $\mu_{\underline{\theta}'}(T')\mid_{[b,\infty)} = T \mid_{[b,\infty)}$ that leaves all arcs in $T' \setminus T'\mid_{[b,\infty)} \subseteq T' \setminus \bigcup_{i \in \ZZ_{>)}}\A(P_i)$ untouched. It follows from Proposition \ref{P:transitivity} that the statement (1) holds for any other triangulation $T'$ with a right fountain at $b$.

	  Item (2) follows from (1) by symmetry.
	  
	  We now show (3): We can pick a sequence $((a_i,b_i))_{i \geq 0}$ in $T$ and a sequence $((a'_i,b'_i))_{i \geq 0}$ in $T'$ such that 
	  \[
	    a'_{i+1} < a_i < a'_i < b'_i < b_i < b'_{i+1}
	  \]
	  for all $i \geq 1$. 
	  For $i \in \ZZ_{>0}$ odd, we set 
	  \[
	    P_i = [a_{i+1},a_{i-1}] \cup [b_{i-1},b_{i+1}].
	  \]
	  These are all finite subpolygons of $T$ and we have $\A(P_i) \cap \s(P_j) = \varnothing$ for odd $i$ and odd $j$ with $i \neq j$. Furthermore, for $i \in \ZZ_{>0}$ odd consider the finite subpolygons 
	  \[
	   P'_i = [a'_{i+1},a'_i] \cup [b'_i,b'_{i+1}] \subseteq P_i
	  \]
	  of $T'$. By Lemma \ref{L:union of subpolygons}, there exists a $T$-admissible sequence $\underline{\theta}$ such that 
	  \[
	   \mu_{\underline{\theta}}(T) \mid_{\bigcup_{i \in \ZZ_{>0} \; \text{odd}}P'_i} = T' \mid_{\bigcup_{i \in \ZZ_{>0} \; \text{odd}}P'_i}.
	  \]
	 Set $\tilde{T} = \mu_{\underline{\theta}}(T)$. Set $Q_0 = [a'_0,b'_0]$ and for $i \in \ZZ_{>0}$ consider the sets 
	 \[
	  Q_i = [a'_i, a'_{i-1}] \cup [b'_{i-1},b'_i].
	 \]

	 Clearly they are finite subpolygons of $T'$. However, they are finite subpolygons of $\tilde{T}$ as well: Indeed, we have $\E(Q_i) \subseteq \E(\IGC) \cup \{(a'_i,b'_i), (a'_{i-1},b'_{i-1})\}$ and $(a'_i,b'_i), (a'_{i-1},b'_{i-1}) \in T' \mid_{P'_{i-1} \cup P'_i} = \tilde{T}\mid_{P'_{i-1} \cup P'_i}$. Furthermore, we have $\A(Q_i) \cap \s(Q_j) = \varnothing$ for all $i \neq j$ and $\bigcup_{i \in \ZZ_{\geq 0}}Q_i = \ZZ \cup \{\pm \infty\}$. By Lemma \ref{L:union of subpolygons} we obtain a $\tilde{T}$-admissible sequence $\underline{\tilde{\theta}}$ with $\mu_{\underline{\tilde{\theta}}}(\tilde{T}) = T'$. By Proposition \ref{P:transitivity} we can precompose the sequence $\underline{\tilde{\theta}}$ with $\underline{\theta}$ to obtain a $T$-admissible sequence $\underline{\gamma}$ with $\mu_{\underline{\gamma}}(T) = T'$. This shows the claim.
	  
	  We now show (4). Assume that both $T$ and $T'$ have a right fountain at $\infty$, the statement can be proved analogously if they have a right fountain at $-\infty$. We can pick a sequence $(\pi_{k_i})_{i \in \ZZ_{\geq 0}}$ from $T$ and a sequence $(\pi_{k'_i})_{i \in \ZZ_{\geq 0}}$ from $T'$ such that for all $i \geq 0$ we have $\pi_{k_i} < \pi_{k'_i} < \pi_{k_{i+1}}$. For $i \in \ZZ_{>0}$ odd consider the sets $P_i = [k_{i-1},k_{i+1}] \cup \{\infty\}$. These are finite subpolygons of $T$ and we have $\A(P_i) \cap \s(P_j) = \varnothing$ for $i \neq j$. Moreover, for $i \geq 0$ odd, the sets $P'_i = [{k'_{i-1}},{k'_{i}}] \cup \{\infty\} \subseteq P_i$ are finite subpolygons of $T'$. By Lemma \ref{L:union of subpolygons} there exists a $T$-admissible sequence $\underline{\theta}$, with $\mu_{\underline{\theta}}(T) \mid_{\bigcup_{i \in \ZZ_{>0} \; \text{odd}}P_i} = T' \mid_{\bigcup_{i \in \ZZ_{>0} \; \text{odd}}P_i}$. Set $\tilde{T} = \mu_{\underline{\theta}}(T)$. The sets $Q_i = [k'_i, {k'_{i+1}}] \cup \{\infty\}$ for $i \geq 0$ are finite subpolygons of both $\tilde{T}$ and $T'$. Furthermore, we have $\A(P_i) \cap \s(P_j) = \varnothing$ for $i \neq j$ and $\bigcup_{i \in \ZZ_{\geq 0}} Q_i = [k'_0, \infty]$. By Lemma \ref{L:union of subpolygons} there exists a $\tilde{T}$-admissible sequence $\underline{\tilde{\theta}}$ with $\mu_{\underline{\tilde{\theta}}}(\tilde{T}) \mid_{[k'_0, \infty]} = T'  \mid_{[k'_0, \infty]}$ and applying \ref{P:transitivity} we can precompose the sequence $\underline{\tilde{\theta}}$ with $\underline{\theta}$ to obtain a $T$-admissible sequence $\underline{\gamma}$ with $\mu_{\underline{\gamma}}(T) = \mid_{[k_0, \infty]} = T'  \mid_{[k_0, \infty]}$. This shows the claim.

	  Item (5) follows from (4) by symmetry.

	\end{proof}

\subsection{Strong mutation equivalence in the infinity-gon}

\begin{theorem}\label{T:strong equiv.cl.}
		Under strong mutation equivalence, every triangulation of the $\infty$-gon $\IG$ belongs to exactly one of the following equivalence classes.
		\begin{itemize}

		\item
			The class $[T_{lf}]$ of locally finite triangulations.

		\item
			The class $[T(a,b)]$ of triangulations with a left fountain at $a$ and a right fountain at $b$ for a unique pair $a,b \in \ZZ$ with $a \leq b$.
			
		\end{itemize}	
	\end{theorem}

	\begin{proof}
		By Theorem \ref{T:classification triangulations of IGC} each triangulation belongs to one of the listed classes.
		It follows directly from Lemma \ref{L:mutation in fountain sections}, Remark \ref{R:fin. subtr.} and Proposition \ref{P:transitivity} that if $T$ and $T'$ are in the same class $[T_{lf}(0)]$ or $[T(a,b)]$ for some
		fixed pair $(a,b)$, then they are strongly mutation equivalent.
		
		Assume now that $T \in [T(a,b)]$ and $T'$ does not have a right fountain at $b$. We use Remark \ref{R:no mutation} in each of the following cases.
		\begin{itemize}
		\item The triangulation $T'$ is locally finite. Then there exists an arc $(i,j) \in T'$ with $i < b < j$, which intersects the infinitely many arcs in the right fountain at $b$ in $T$, so we cannot have $T \leq_s T'$.
				
		\item The triangulation $T'$ has a right fountain at $b'<b$. Then there is an arc $(b',k) \in T'$ with $b < k$, which intersects the infinitely many arcs in the right fountain at $b$ in $T$, so we cannot have $T \leq_s T'$.
		
		\item The triangulation $T'$ has a right fountain at $b'>b$. Then by the previous bullet point we cannot have $T' \leq_s T$.
		\end{itemize}

		By symmetry it follows that $T$ and $T'$ are not strongly mutation equivalent if $T'$ does not have a left fountain at $a$. Therefore, $T$ and $T'$ are not mutation equivalent if they do not belong to the same class 
		$[T_{lf}]$ or $[T(a,b)]$ for some fixed pair $(a,b)$.
		
	\end{proof}

\begin{rmq}\label{R:representants IG}
We can pick representatives of each of the strong equivalence classes of $\IG$ as follows:
			\begin{itemize}
				\item{In $[T_{lf}]$ (cf.\ Figure \ref{fig:locallyfinite}): \[
							t_{lf} = \{(-k,k) \mid k \in \mathbb{Z}_{>0}\} \cup \{(-k,k+1) \mid k \in \mathbb{Z}_{>0}\}.
				                     \]
 
				}
				\item{In $[T(a,b)]$ (cf.\ Figure \ref{fig:fountain}):
				\[
				 t(a,b) = \{(k,a) \mid k \in \mathbb{Z}_{\leq a-2}\} \cup \{(b,k) \mid k \in \mathbb{Z}_{\geq b+2}\} \cup \{(a,k) \mid a+2 \leq k \leq b\}.
				\]
				}

			\end{itemize}
		
\end{rmq}

\begin{figure}
\begin{center}
\begin{tikzpicture}[scale =.65]
\tikzstyle{every node}=[font=\small]
\draw (-7.5,0) -- (7.5,0);
\draw[dotted] (-8.5,0) -- (-7.5,0);
\draw[dotted] (7.5,0) -- (8.5,0);
\draw (0,0.1) -- (0,-0.1) node[below]{$0$};
\draw (1,0.1) -- (1,-0.1) node[below]{$1$};
\draw (-1,0.1) -- (-1,-0.1) node[below]{$-1$};
\draw (2,0.1) -- (2,-0.1) node[below]{$\ldots$};
\draw (-2,0.1) -- (-2,-0.1) node[below]{$\ldots$};
\draw (3,0.1) -- (3,-0.1); 
\draw (4,0.1) -- (4,-0.1); 
\draw (-3,0.1) -- (-3,-0.1) node[below]{$a$};
\draw (5,0.1) -- (5,-0.1) node[below]{$b$};
\draw (6,0.1) -- (6,-0.1);
\draw (-4,0.1) -- (-4,-0.1);
\draw (-5,0.1) -- (-5,-0.1);
\draw (-6,0.1) -- (-6,-0.1);

\path (-1,0) edge [out= 60, in= 120] (1,0);
\path (-2,0) edge [out= 60, in= 120] (2,0);
\path (-3,0) edge [out= 60, in= 120] (3,0);
\path (-4,0) edge [out= 60, in= 120] (4,0);
\path (-5,0) edge [out= 60, in= 120] (5,0);
\path (-6,0) edge [out= 60, in= 120] (6,0);

\path (-1,0) edge [out= 60, in= 120] (2,0);
\path (-2,0) edge [out= 60, in= 120] (3,0);
\path (-3,0) edge [out= 60, in= 120] (4,0);
\path (-4,0) edge [out= 60, in= 120] (5,0);
\path (-5,0) edge [out= 60, in= 120] (6,0);

\node (a) at (7,1) {$\ldots$};
\node (a) at (-7,1) {$\ldots$};

\end{tikzpicture}
\end{center}
\caption{The triangulation $t_{lf}$ of $\IG$}\label{fig:locallyfinite}
\end{figure}

\begin{figure}
\begin{center}
\begin{tikzpicture}[scale =.65]
\tikzstyle{every node}=[font=\small]
\draw (-7.5,0) -- (7.5,0);
\draw[dotted] (-8.5,0) -- (-7.5,0);
\draw[dotted] (7.5,0) -- (8.5,0);
\draw (0,0.1) -- (0,-0.1);
\draw (1,0.1) -- (1,-0.1)node[below]{$b$};
\draw (-1,0.1) -- (-1,-0.1);
\draw (2,0.1) -- (2,-0.1) ;
\draw (-2,0.1) -- (-2,-0.1) node[below]{$a$};
\draw (3,0.1) -- (3,-0.1);
\draw (4,0.1) -- (4,-0.1); 
\draw (-3,0.1) -- (-3,-0.1);
\draw (5,0.1) -- (5,-0.1);
\draw (6,0.1) -- (6,-0.1);
\node (a) at (-7,1) {$\ldots$};
\node (a) at (-7,1) {$\ldots$};
\draw (-4,0.1) -- (-4,-0.1);
\draw (-5,0.1) -- (-5,-0.1);
\draw (-6,0.1) -- (-6,-0.1);

\path (1,0) edge [out= 60, in= 120] (3,0);
\path (1,0) edge [out= 60, in= 120] (4,0);
\path (1,0) edge [out= 60, in= 120] (5,0);
\path (1,0) edge [out= 60, in= 120] (6,0);

\path (-4,0) edge [out= 60, in= 120] (-2,0);
\path (-5,0) edge [out= 60, in= 120] (-2,0);
\path (-6,0) edge [out= 60, in= 120] (-2,0);

\node (a) at (7,1) {$\ldots$};
\node (a) at (-7,1) {$\ldots$};
\path (-2,0) edge [out= 60, in= 120] (1,0);
\path (-2,0) edge [out= 60, in= 120] (0,0);

\node (a) at (7,1) {$\ldots$};
\node (a) at (-7,1) {$\ldots$};

\end{tikzpicture}
\end{center}
\caption{The triangulation $t(a,b)$ of $\IG$}\label{fig:fountain}
\end{figure}

\subsection{Strong mutation equivalence in the completed infinity-gon}

We first provide a classification of the strong mutation equivalence classes of triangulations of the completed infinity-gon. The preorder $\leq_s$ induces a partial order on the set of strong mutation equivalence classes of triangulations of the completed infinity-gon. At the end of this section we describe the structure of the Hasse diagram of this poset. 

\begin{theorem}\label{T:strong equiv.cl.completed}
		Under strong mutation equivalence, every triangulation of $\IGC$ belongs to exactly one of the following equivalence classes.
		\begin{itemize}

		\item
			The class $[T_{lf}]$ of locally finite triangulations.

		\item
			The class $[T(a,b)]$ of triangulations with a left fountain at $a$ and a right fountain at $b$ for a unique pair $(a,b)$ with $a,b \in \mathbb{Z} \cup \{\pm \infty\}$ and $a \leq b$ or $a = \infty, b \in \mathbb{Z}$ or $a \in \mathbb{Z}, b = -\infty$.
			
		\end{itemize}	
	\end{theorem}

	\begin{proof}
	
	If $T, T' \in [T_{lf}]$ they are mutation equivalent by Lemma \ref{L:mutation in fountain sections}(3). 
	
	If $T,T' \in [T(a,b)]$ for $a, b \in \ZZ$, then we have $\pi_b, \alpha_a \in T \cap T'$. By Lemma \ref{L:mutation in fountain sections} (1) and (2), using Proposition \ref{P:transitivity} we have a $T$-admissible sequence $\underline{\theta}$ with $\mu_{\underline{\theta}}(T) \mid_{(-\infty,a] \cup [b,\infty)} = T' \mid_{(-\infty,a] \cup [b,\infty)}$ that leaves all other arcs in $T$ untouched. If $a=b$ then we are done, and otherwise the set $P=[a,b] \cup \{\pm \infty\}$ is a finite subpolygon of both $\mu_{\underline{\theta}}(T)$ and $T'$ and there exists a $\mu_{\underline{\theta}}(T)$-admissible sequence $\underline{\theta}'$ with $\mu_{\underline{\theta}'} \circ \mu_{\underline{\theta}}(T) = T'$. The statement follows by Proposition \ref{P:transitivity}.
	
	Assume now that $T, T' \in [T(a,b)]$ for $a \in \ZZ$ and $b = \infty$ (respectively $b = -\infty$). By Lemma \ref{L:mutation in fountain sections} (1) and (4), using Proposition \ref{P:transitivity} we have a $T$-admissible sequence $\underline{\theta}$ and a $k \in \ZZ$, such that $\pi_k \in T'$ (respectively $\alpha_k \in T'$) with $\mu_{\underline{\theta}}(T) \mid_{(-\infty,a] \cup [k,\infty) \cup \{b\}} = T' \mid_{(-\infty,a] \cup [k,\infty) \cup \{b\}}$ that leaves all other arcs in $T$ untouched. If $k = a$ then we are done, since if two triangulations agree on the set $(-\infty,\infty]$ then they must on all of $[-\infty,\infty]$. Otherwise, we have $a < k$ and the set $[a,k] \cup \{-\infty\} \cup \{b\}$ is a finite subpolygon of both $\mu_{\underline{\theta}}(T)$ and $T'$ and it follows as above that $T$ and $T'$ are mutation equivalent.
	
	With an analogous argument we can show that if $T, T' \in [T(a,b)]$ with $a \in \{\pm \infty\}$ and $b \in \ZZ$, respectively with $a,b \in \{\pm \infty\}$, then they are mutation equivalent.
	
	The rest of the proof follows similarly to the proof of Theorem \ref{T:strong equiv.cl.} by applying Remark \ref{R:no mutation}.

	\end{proof}

\begin{rmq}\label{R:representants}
We can pick representatives of each of the strong equivalence classes of $\IGC$ as follows. Recall that we omit the generic curve for brevity. 
			\begin{itemize}
				\item{For $[T_{lf}]$:
				$t_{lf} = \{(-k,k) \mid k \in \mathbb{Z}_{>0}\} \cup \{(-k,k+1) \mid k \in \mathbb{Z}_{>0}\}$.}

					\item For $[T(a,b)]$ with $a,b \in \ZZ$ and $a \leq b$:
					\begin{eqnarray*}
						\overline{t}(a,b) & = & \{(k,a) \mid k \in \mathbb{Z}_{\leq a-2}\} \; \cup \; \{\alpha_a\} \cup \\&& \{(b,k) \mid k \in \mathbb{Z}_{\geq b+2}\} \; \cup \; \{\pi_k \mid a \leq k \leq b\}.
					\end{eqnarray*}

					\item For $[T(-\infty,b)]$ with $b \in \ZZ$ (cf.\ Figure \ref{fig:-infty,b}):
					\[
						\overline{t}(-\infty,b) = \{\alpha_k \mid k \leq b\} \cup \{(b,k) \mid k \in \mathbb{Z}_{\geq b+2}\} \cup \{\pi_b\}.
					\]

					\item For $[T(\infty,b)]$ with $b \in \ZZ$ (cf.\ Figure \ref{fig:infty,b}):
					\[
						\overline{t}(\infty,b) = \{\pi_k \mid k \leq b\} \cup \{(b,k) \mid k \in \mathbb{Z}_{\geq b+2}\}.
					\]

					\item For $[T(a, \infty)]$ with $a \in \ZZ$:
					\[
						\overline{t}(a, \infty)= \{(k,a) \mid k \in \mathbb{Z}_{\leq a-2}\} \cup \{\alpha_a\} \cup \{\pi_k \mid k \geq a\}.
					\]

					\item For $[T(a, -\infty)]$ with $a \in \ZZ$:
					\[
						\overline{t}(a, -\infty) = \{(k,a) \mid k \in \mathbb{Z}_{\leq a-2}\}  \cup \{\alpha_k \mid k \geq a\}.
					\]

					\item For $[T(-\infty,\infty)]$ (cf.\ Figure \ref{fig:-infty,infty}):
					\[
						\overline{t}(-\infty,\infty) = \{\alpha_k \mid k \leq 0\} \cup \{\pi_k \mid k \geq 0\}.
					\]
					
					\item For $[T(-\infty,-\infty)]$, respectively $[T(\infty,\infty)]$:
					\[
						\overline{t}(-\infty,-\infty) = \{\alpha_k \mid k \in \mathbb{Z}\}, \; \; \text{respectively} \; \; \overline{t}(\infty,\infty) = \{\pi_k \mid k \in \mathbb{Z}\}.
					\]

	\end{itemize}
\end{rmq}

\begin{figure}
\begin{center}
\begin{tikzpicture}[scale =.65]
\tikzstyle{every node}=[font=\small]
\draw (-7.5,0) -- (8.5,0);
\draw[dotted] (-8.5,0) -- (-5.5,0);
\draw[dotted] (8.5,0) -- (9.5,0);

\draw (0,0.1) -- (0,-0.1);
\draw (1,0.1) -- (1,-0.1);
\draw (-1,0.1) -- (-1,-0.1);
\draw (2,0.1) -- (2,-0.1) ;
\draw (-2,0.1) -- (-2,-0.1) node[below]{$b$};
\draw (3,0.1) -- (3,-0.1);
\draw (4,0.1) -- (4,-0.1); 
\draw (5,0.1) -- (5,-0.1);
\draw (6,0.1) -- (6,-0.1);
\node (a) at (-7,1) {$\ldots$};
\node (a) at (7,1) {$\ldots$};
\draw (-4,0.1) -- (-4,-0.1);
\draw (-5,0.1) -- (-5,-0.1);
\draw (-6,0.1) -- (-6,-0.1);

\draw (-3,0.1) -- (-3,-0.1);
\node (-infty) at (-8,2.5) {$\bullet$};
\node (-infty) at (-9,2.5) {$-\infty$};
\node (infty) at (8,2.5) {$\bullet$};
\node (infty) at (9,2.5) {$\infty$};

\path (-2,0) edge [out= 60, in= 120] (0,0);
\path (-2,0) edge [out= 60, in= 120] (1,0);
\path (-2,0) edge [out= 60, in= 120] (2,0);
\path (-2,0) edge [out= 60, in= 120] (3,0);
\path (-2,0) edge [out= 60, in= 120] (4,0);
\path (-2,0) edge [out= 60, in= 120] (5,0);
\path (-2,0) edge [out= 60, in= 120] (6,0);
\draw (8,2.5) -- (0,2.5) to[out=180,in = 90] (-2,0);
\draw (-8,2.5) -- (-4,2.5) to[out=0,in = 90] (-2,0);
\draw (-8,2.5) -- (-5,2.5) to[out=0,in = 90] (-3,0);
\draw (-8,2.5) -- (-6,2.5) to[out=0,in = 90] (-4,0);
\draw (-8,2.5) -- (-7,2.5) to[out=0,in = 90] (-5,0);
\draw (-8,2.5) -- (-8,2.5) to[out=0,in = 90] (-6,0);

\end{tikzpicture}
\end{center}
\caption{The triangulation $\overline{t}(-\infty,b)$ of $\IGC$}\label{fig:-infty,b}
\end{figure}

\begin{figure}
\begin{center}
\begin{tikzpicture}[scale =.65]
\tikzstyle{every node}=[font=\small]
\draw (-7.5,0) -- (8.5,0);
\draw[dotted] (-8.5,0) -- (-5.5,0);
\draw[dotted] (8.5,0) -- (9.5,0);

\draw (0,0.1) -- (0,-0.1);
\draw (1,0.1) -- (1,-0.1);
\draw (-1,0.1) -- (-1,-0.1);
\draw (2,0.1) -- (2,-0.1) ;
\draw (-2,0.1) -- (-2,-0.1) node[below]{$b$};
\draw (3,0.1) -- (3,-0.1);
\draw (4,0.1) -- (4,-0.1); 
\draw (5,0.1) -- (5,-0.1);
\draw (6,0.1) -- (6,-0.1);
\node (a) at (-7,1) {$\ldots$};
\node (a) at (7,1) {$\ldots$};
\draw (-4,0.1) -- (-4,-0.1);
\draw (-5,0.1) -- (-5,-0.1);
\draw (-6,0.1) -- (-6,-0.1);

\draw (-3,0.1) -- (-3,-0.1);
\node (-infty) at (-8,2.5) {$\bullet$};
\node (-infty) at (-9,2.5) {$-\infty$};
\node (infty) at (8,2.5) {$\bullet$};
\node (infty) at (9,2.5) {$\infty$};

\path (-2,0) edge [out= 60, in= 120] (0,0);
\path (-2,0) edge [out= 60, in= 120] (1,0);
\path (-2,0) edge [out= 60, in= 120] (2,0);
\path (-2,0) edge [out= 60, in= 120] (3,0);
\path (-2,0) edge [out= 60, in= 120] (4,0);
\path (-2,0) edge [out= 60, in= 120] (5,0);
\path (-2,0) edge [out= 60, in= 120] (6,0);
\draw (8,2.5) -- (0,2.5) to[out=180,in = 90] (-2,0);
\draw (8,2.5) -- (-1,2.5) to[out=180,in = 90] (-3,0);
\draw (8,2.5) -- (-2,2.5) to[out=180,in = 90] (-4,0);
\draw (8,2.5) -- (-3,2.5) to[out=180,in = 90] (-5,0);
\draw (8,2.5) -- (-4,2.5) to[out=180,in = 90] (-6,0);

\end{tikzpicture}
\end{center}
\caption{The triangulation $\overline{t}(\infty,b)$ of $\IGC$}\label{fig:infty,b}
\end{figure}

\begin{figure}
\begin{center}
\begin{tikzpicture}[scale =.65]
\tikzstyle{every node}=[font=\small]
\draw (-7.5,0) -- (8.5,0);
\draw[dotted] (-8.5,0) -- (-5.5,0);
\draw[dotted] (8.5,0) -- (9.5,0);

\draw (0,0.1) -- (0,-0.1) node[below]{$0$};
\draw (1,0.1) -- (1,-0.1);
\draw (-1,0.1) -- (-1,-0.1);
\draw (2,0.1) -- (2,-0.1) ;
\draw (-2,0.1) -- (-2,-0.1);
\draw (3,0.1) -- (3,-0.1);
\draw (4,0.1) -- (4,-0.1); 
\draw (5,0.1) -- (5,-0.1);
\draw (6,0.1) -- (6,-0.1);
\node (a) at (-7,1) {$\ldots$};
\node (a) at (7,1) {$\ldots$};
\draw (-4,0.1) -- (-4,-0.1);
\draw (-5,0.1) -- (-5,-0.1);
\draw (-6,0.1) -- (-6,-0.1);

\draw (-3,0.1) -- (-3,-0.1);
\node (-infty) at (-8,2.5) {$\bullet$};
\node (-infty) at (-9,2.5) {$-\infty$};
\node (infty) at (8,2.5) {$\bullet$};
\node (infty) at (9,2.5) {$\infty$};

\draw (-8,2.5) -- (-2,2.5) to[out=0,in = 90] (0,0);
\draw (-8,2.5) -- (-3,2.5) to[out=0,in = 90] (-1,0);
\draw (-8,2.5) -- (-4,2.5) to[out=0,in = 90] (-2,0);
\draw (-8,2.5) -- (-5,2.5) to[out=0,in = 90] (-3,0);
\draw (-8,2.5) -- (-6,2.5) to[out=0,in = 90] (-4,0);
\draw (-8,2.5) -- (-7,2.5) to[out=0,in = 90] (-5,0);
\draw (-8,2.5) -- (-8,2.5) to[out=0,in = 90] (-6,0);

\draw (8,2.5) -- (2,2.5) to[out=180,in = 90] (0,0);
\draw (8,2.5) -- (3,2.5) to[out=180,in = 90] (1,0);
\draw (8,2.5) -- (4,2.5) to[out=180,in = 90] (2,0);
\draw (8,2.5) -- (5,2.5) to[out=180,in = 90] (3,0);
\draw (8,2.5) -- (6,2.5) to[out=180,in = 90] (4,0);
\draw (8,2.5) -- (7,2.5) to[out=180,in = 90] (5,0);
\draw (8,2.5) -- (8,2.5) to[out=180,in = 90] (6,0);

\end{tikzpicture}
\end{center}
\caption{The triangulation $\overline{t}(-\infty,\infty)$ of $\IGC$}\label{fig:-infty,infty}
\end{figure}

\begin{prop}\label{P:Hasse diagram}
 The preorder $\leq_s$ induces a partial order on the set of strong mutation equivalence classes of triangulations of the completed infinity-gon. The graph from Figure \ref{fig:Hasse diagram} is, for each $a,b \in \ZZ$ with $a \leq b$, a subdiagram of the Hasse diagram of this poset.
\end{prop}

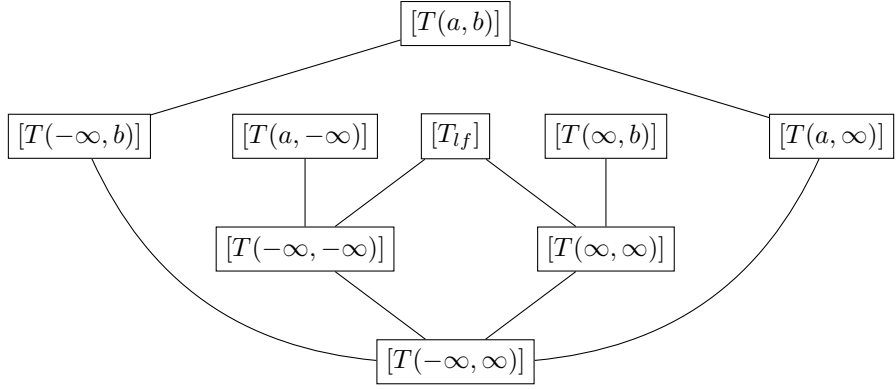
\begin{figure}
\begin{center}
 \begin{tikzpicture}
  \node[draw, rectangle] (0) at (0,0) {$[T(-\infty, \infty)]$};
  \node[draw, rectangle] (1a) at (-2,1.5) {$[T(-\infty,-\infty)]$};
  \node[draw, rectangle] (1b) at (2,1.5) {$[T(\infty,\infty)]$};
  \node[draw, rectangle] (2a) at (-5,3) {$[T(-\infty,b)]$};
  \node[draw, rectangle] (2b) at (5,3) {$[T(a, \infty)]$};
  \node[draw, rectangle] (3a) at (-2,3) {$[T(a, -\infty)]$};
  \node[draw, rectangle] (3b) at (2,3) {$[T(\infty,b)]$};
  \node[draw, rectangle] (4) at (0,3) {$[T_{lf}]$};
  \node[draw, rectangle] (5) at (0,4.5) {$[T(a,b)]$};
  
  \draw (0) -- (1a);
  \draw (0) -- (1b);
  \draw (1a) -- (3a);
  \draw (1b) -- (3b);
  \draw (1a) -- (4);
  \draw (1b) -- (4);
  \draw (0) edge[bend left] (2a);
  \draw (0) edge[bend right] (2b);
  \draw (2a) -- (5);
  \draw (2b) -- (5);
  
 \end{tikzpicture}
 \end{center}
 \caption{For $a \leq b$ with $a,b \in \mathbb{Z}$ this forms a subdiagram of the Hasse diagram of strong mutation equivalence classes of triangulations of $\IGC$ with respect to the order $\leq_s$}\label{fig:Hasse diagram}
 \end{figure}

\begin{proof}
 Clearly the relation $\leq_s$ describes a partial order on the set of strong mutation equivalence classes. To show that our diagram is a subgraph of the Hasse diagram of this poset, using the notation from \ref{R:representants} we pick representatives and explicitely write down admissible sequences along which we can mutate one into another. 
 
 Setting $\underline{\theta}_1 = (\alpha_{-i})_{i \geq 0}$ and $\underline{\theta}_2 = (\pi_{i})_{i \geq 0}$ we have 
  \[
   \mu_{\underline{\theta}_1}(\overline{t}(-\infty,\infty)) = \overline{t}(\infty,\infty) \; \; \text{and} \; \; \mu_{\underline{\theta}_2}(\overline{t}(-\infty,\infty)) = \overline{t}(-\infty,-\infty).
  \]
 Let $t_a = \{\alpha_k \mid k \leq a\} \cup \{\pi_k \mid k \geq a\}$ and $t_b = \{\alpha_k \mid k \leq b\} \cup \{\pi_k \mid k \geq b\}$. We have $t_a, t_b \in [T(-\infty,\infty)]$.
  Setting $\underline{\theta}_3 = (\alpha_{-i})_{i \geq a-1}$ and $\underline{\theta}_4 = (\pi_{i})_{i \geq b+1}$ we have 
  \[
   \mu_{\underline{\theta}_3}(t_a) = \overline{t}(a,\infty) \; \; \text{and} \; \; \mu_{\underline{\theta}_4}(t_b) = \overline{t}(-\infty,b).
  \]
  Setting $\underline{\theta}_5 = (\pi_0, (\pi_i, \pi_{-i})_{i \geq 1})$ and $\underline{\theta}_6 = (\alpha_{0}, (\alpha_i, \alpha_{-i})_{i \geq 0})$ we have 
  \[
   \mu_{\underline{\theta}_5}(\overline{t}(\infty,\infty)) = t_{lf} \; \; \text{and} \; \; \mu_{\underline{\theta}_6}(\overline{t}(-\infty,-\infty)) = t_{lf}.
  \]
 Setting $\underline{\theta}_7 = (\pi_{b+i})_{i \geq 1}$ and $\underline{\theta}_8 = (\alpha_{a-i})_{i \geq 1 }$ we have 
  \[
   \mu_{\underline{\theta}_7}(\overline{t}(\infty,\infty)) = \overline{t}(\infty,b) \; \; \text{and} \; \; \mu_{\underline{\theta}_8}(\overline{t}(-\infty,-\infty)) = \overline{t}(a,-\infty).
  \]
  Setting $t = \{\alpha_k \mid k \leq a\} \cup \{\pi_k \mid a \leq k \leq b\} \mid \{(b,k) \mid k \geq b+2\} \in [T(-\infty,b)]$, we get
  \[
    \mu_{\underline{\theta}_7}(\overline{t}(a,\infty)) = \overline{t}(a,b) \; \; \text{and} \; \; \mu_{\underline{\theta}_8}(t) = \overline{t}(a,b).
  \]

\end{proof}

In fact, it is straightforward to check that there are no other edges in the Hasse diagram, using Remark \ref{R:no mutation}. As we will not need this in the rest of the paper, we leave this as an exercise to the interested reader.

\section{Completed mutations}

The restriction when solely considering mutations along admissible sequences is twofold: First, not all triangulations of $\IGC$ are strongly mutation equivalent and second, mutating a triangulation along an admissible sequence does not in general yield a triangulation.
In this section we fix the latter issue by providing a method to complete the mutation $\mu_{\underline{\theta}}(T)$ of a triangulation $T$ along a $T$-admissible sequence $\underline{\theta}$ to a triangulation.

\begin{lem}\label{L:finite parts}
		Let $T$ be a triangulation of $\IGC$ and let $\underline{\theta} = (\theta_i)_{i \in I}$ be a $T$-admissible sequence. If $(m,l) \in \mu_{\underline{\theta}}(T)$ is a peripheral arc, then the set of arcs 
			\[
				\mu_{\underline{\theta}}(T)\mid_{[m,l]}=\{(a,b) \in \mu_{\underline{\theta}}(T) \mid m \leq a < b \leq l\}
			\]
		with endpoints in $[m,l]$ is a triangulation of the polygon with endpoints $m, m+1, \ldots, l$.
	\end{lem}
	
	\begin{proof}
		For notational simplicity, for any $k \in \mathbb{Z}$ set $T_k = \mu_{\theta_{k}} \circ \ldots \circ \mu_{\theta_1}(T)$.
		There exists a $k_1 \in \mathbb{Z}_{>0}$ such that $(m,n) \in T_l$ for all $l \geq k_1$. There are finitely many arcs 
		$\gamma_1, \ldots \gamma_j$ in the subtriangulation $T_{k_1}\mid_{[m,l]}$ of $T_{k_1}$. Thus there exists a $k_2 \geq k_1$ such that for all $l \geq k_2$ and all $1 \leq i \leq j$ we have
			\[
				\mu_{\theta_l} \circ \ldots \circ \mu_{\theta_1}(\gamma_i) = \mu_{\theta_{k_2}} \circ \ldots \circ \mu_{\theta_1}(\gamma_i).
			\]
		Therefore $\mu_{\underline{\theta}}(T)\mid_{[m,l]} =T_{k_2}\mid_{[m,l]}$, and since $T_{k_2}$ is a triangulation with finite subpolygon $[m,l]$ by Remark \ref{R:subtriangulation} this proves the claim.
	\end{proof}

 We now provide a method to complete a mutated triangulation by adding arcs until we obtain a triangulation. Such a completion is by no means unique and we could just complete by randomly adding arcs that do not intersect any of the arcs already contained in our mutated triangulation. However, the existence of strictly asymptotic arcs in $\IGC$ lends itself to a somewhat natural completion via Pr\"ufer curves and adic curves.  We use the following auxiliary sets:
	\begin{eqnarray*}
		\mathcal{P}(\mu_{\underline{\theta}}(T)) &=& \{\pi_k \text{ Pr\"ufer curve} \mid \pi_k \text{ intersects no arc in } \mu_{\underline{\theta}}(T)\} \\ 
		\mathcal{A}(\mu_{\underline{\theta}}(T)) &=& \{\alpha_k \text{ adic curve} \mid \alpha_k \text{ intersects no arc in } \mu_{\underline{\theta}}(T)\} \\ 
		\tilde{\mathcal{P}}(\mu_{\underline{\theta}}(T)) &=& \{\pi_k \text{ Pr\"ufer curve} \mid \pi_k \text{ intersects no arc in } \mu_{\underline{\theta}}(T) \cup \mathcal{A}						(\mu_{\underline{\theta}}(T))\} \\ 
		\tilde{\mathcal{A}}(\mu_{\underline{\theta}}(T)) &=& \{\alpha_k \text{ adic curve} \mid \alpha_k \text{ intersects no arc in } \mu_{\underline{\theta}}(T) \cup \mathcal{P}						(\mu_{\underline{\theta}}(T))\}.
	\end{eqnarray*}
	
	\begin{defi}
		We call the set of arcs
			\[
				\overline{\mu_{\underline{\theta}}(T)}^P = \mu_{\underline{\theta}}(T) \cup \mathcal{P}(\mu_{\underline{\theta}}(T)) \cup \tilde{\mathcal{A}}(\mu_{\underline{\theta}}(T))
			\]
		the {\em Pr\"ufer-completion of $\mu_{\underline{\theta}}(T)$}. Analogously, we call the set of arcs 
			\[
				\overline{\mu_{\underline{\theta}}(T)}^a = \mu_{\underline{\theta}}(T) \cup \mathcal{A}(\mu_{\underline{\theta}}(T)) \cup \tilde{\mathcal{P}}(\mu_{\underline{\theta}}(T))
			\]
		the {\em adic completion of $\mu_{\underline{\theta}}(T)$}.
	\end{defi}
	
	\begin{rmq}
	 In general the Pr\"ufer and adic completion do not coincide.
	\end{rmq}

From now on we only consider Pr\"ufer completions. Adic completions are the dual concept and all of the following results hold for adic completions by symmetry. From now on we write 
$\overline{\mu_{\underline{\theta}}}(T) = \overline{\mu_{\underline{\theta}}(T)}^P$ and call it the {\em completed mutation of $T$ along $\underline{\theta}$.}

	\begin{theorem}\label{T:completions are triangulations}
		Let $T$ be a triangulation of $\IGC$ and let $\underline{\theta}$ be a $T$-admissible sequence. Then the completed mutation of $T$ along $\underline{\theta}$ is a triangulation of $\IGC$. 
	\end{theorem}
	
	\begin{proof}
		Assume that an arc $\gamma$ intersects no arc in $\overline{\mu_{\underline{\theta}}}(T)$. We will show that then $\gamma$ itself must lie in 
		$\overline{\mu_{\underline{\theta}}}(T)$. If $\gamma$ is a Pr\"ufer curve, then $\gamma \in \mathcal{P}(\mu_{\underline{\theta}}(T))$ and if it is an adic curve, then 
		$\gamma \in \tilde{\mathcal{A}}(\mu_{\underline{\theta}}(T))$, 
		thus in particular $\gamma$ lies in $\overline{\mu_{\underline{\theta}}}(T)$.
		
		Assume thus that $\gamma = (m,l)$ with $m<l \in \mathbb{Z}$ is a peripheral arc. If there exists a peripheral arc $(m',l') \in \mu_{\underline{\theta}}(T)$ with $m' \leq m < l \leq l'$, 		
		then by Lemma \ref{L:finite parts} we have
		$\gamma = (m,l) \in \mu_{\underline{\theta}}(T) \subseteq \overline{\mu_{\underline{\theta}}}(T)$ and we are done. 
		
		Otherwise, if there exists no such arc $(m',l')$, it is straight-forward to check that we have $\pi_m, \pi_l \in \mathcal{P}(\mu_{\underline{\theta}}(T))$ or $\alpha_m, \alpha_l \in \tilde{\mathcal{A}}(\mu_{\underline{\theta}}(T))$. Without loss of generality assume the former is the case. 
		Set 
			\[
				n = \max \{j \mid (m,j) \in \mu_{\underline{\theta}}(T)\}.
			\] 
		Observe that the set over which we take the maximum is not empty: Because $\gamma$ intersects no arc in $\overline{\mu_{\underline{\theta}}}(T)$, we have 
		$\pi_{m+1} \notin \mathcal{P}(\mu_{\underline{\theta}}(T))$, therefore there is an arc 
		$(a,b) \in \mu_{\underline{\theta}}(T)$ that intersects $\pi_{m+1}$ but not $(m,l)$ nor $\pi_m$ nor $\pi_l$, so $m \leq a < m+1 < b \leq l$, and thus 
		$(a,b) = (m,b) \in \mu_{\underline{\theta}}(T)$. 
		
		Assume as a contradiction that $n \neq l$. By the same argument as above for $n$ instead of $m+1$, there is an arc $(a,b) \in \mu_{\underline{\theta}}(T)$ with 
		$m \leq a < n < b \leq l$. However, if $m < a$ then this would imply that $(m,n)$ and $(a,b)$ intersect, contradicting the assumption. Therefore we have $a = m$ and 
		$(m,b) \in \mu_{\underline{\theta}}(T)$ contradicting the maximality of $n$. Thus in fact
		we must have $n = l$ and $(m,l) \in \mu_{\underline{\theta}}(T) \subseteq \overline{\mu_{\underline{\theta}}}(T)$.
	\end{proof}

\begin{rmq}\label{R:one limit point completed}
 In the combinatorial model of the $\infty$-gon with only one limit point at $\infty$ (cf.\ Remark \ref{R:one limit point}) we can define a unique completion: Assume that $T$ is a triangulation of the $\infty$-gon with one added point at $\infty$ and let $\underline{\theta}$ be a $T$-admissible sequence. Denoting the arc connecting a point $a \in \ZZ$ with $\infty$ by $(a,\infty)$, we define the {\em completed mutation of $T$ along $\underline{\theta}$} to be
 \[
  \overline{\mu_{\underline{\theta}}}(T) = \mu_{\underline{\theta}}(T) \cup \{(a, \infty) \mid a \in \ZZ \;\text{and $(a,\infty)$ does not intersect any arc in}\; \mu_{\underline{\theta}}(T)\}.
 \]
 This is a triangulation of the $\infty$-gon with one point at $\infty$; this follows analogously to Theorem \ref{T:completions are triangulations}.
 
 Because of its links with the representation theory of the polynomial ring, we are in particular interested in our example of the completed $\infty$-gon $\IGC$ where 
we have  two limit points at $\pm \infty$. Note however that, with minor adaptations, all statements in the rest of this paper hold for the $\infty$-gon with one limit point at $\infty$.
\end{rmq}

An important example of completed mutations is moving a right fountain one step to the right, and dually, moving a left fountain one step to the left.

\begin{lem}\label{L:pushing}
 Let $a, b \in \ZZ$ and let $T \in [T(a,b)]$ be a triangulation of $\IGC$. Then there exist $T$-admissible sequences $\underline{a}^-$ and $\underline{b}^+$ such that $\overline{\mu_{\underline{a}^-}}(T) \in [T(a-1,b)]$ and $\overline{\mu_{\underline{b}^+}}(T) \in [T(a,b+1)]$.
\end{lem}

\begin{proof}
 The triangulation $T$ is strongly mutation equivalent to the triangulation $\overline{t}(a,b)$ from Remark \ref{R:representants}. Assume that $\underline{\theta}$ is a $T$-admissible sequence with $\mu_{\underline{\theta}_1}(T) = \overline{t}(a,b)$. Consider now the $\overline{t}(a,b)$-admissible sequence $\underline{\alpha} = ((a-k,a))_{k \geq 2}$; we have $\overline{\mu_{\underline{\alpha}}}(\overline{t}(a,b)) = \overline{t}(a-1,b)$. It follows by Proposition \ref{P:transitivity} that \[\overline{\mu_{\underline{\theta} \cup \underline{\alpha}}}(T) = \overline{t}(a-1,b) \in [T(a-1,b)].\] Symmetrically, considering the $\overline{t}(a,b)$-admissible sequence $\underline{\beta}=((b,b+k))_{k \geq 2}$ we have \[\overline{\mu_{\underline{\theta} \cup \underline{\beta}}}(T) = \overline{t}(a,b+1) \in [T(a,b+1)].\]
\end{proof}

\section{Transfinite mutations}

Completed mutations along admissible sequences provide new connections between triangulations of $\IGC$.
However, we cannot pass freely between strong mutation equivalence classes of triangulations of $\IGC$ via completed mutations. This can be fixed if we consider admissible compositions of admissible sequences.

\begin{defi}
	Let $T$ be a triangulation of $\IGC$. We call a sequence $\overline{\theta} = (\underline{\theta^i})_{i \in I}$ of admissible sequences (where throughout this paper we assume $I = \{1, \ldots, n\}$ or $I = \mathbb{Z}_{> 0}$) a
	{\em $T$-admissible composition of completed mutations}, if, setting $T_1 = T$, for all $i \in I$ the sequence $\underline{\theta^i}$ is $T_i$-admissible, where for $i \geq 1$ we set
		\begin{eqnarray*}
			T_{i+1} &=& \overline{\mu_{\underline{\theta^{i}}}}(T_{i}).
		\end{eqnarray*}
	The {\em transfinite mutation of $T$ along $\overline {\theta}$} is the set
	\[
	\mu_{\overline{\theta}}(T)  = \bigcup_{i \in I}\{\gamma \in T_i \mid \underline{\theta}^k \; \text{leaves $\gamma$ untouched for all}\; k \geq i\}.
	\]
\end{defi}

\begin{rmq}\label{R:transf.nc}
 A transfinite mutation of a triangulation of $\IGC$ consists of mutually non-intersecting arcs. Indeed, with the notation as above, if $\alpha, \beta \in \mu_{\overline{\theta}}(T)$ then there exists a $k \in I$ such that $\alpha, \beta \in T_k$, which is a triangulation.
 
 However, a transfinite mutation of a triangulation is not necessarily a triangulation. Indeed, a $T$-admissible sequence can be interpreted as a $T$-admissible composition of completed mutations of length one, and we already know from Example \ref{E:not a triangulation}, that the mutation of a triangulation along an admissible sequence is not necessarily a triangulation.
\end{rmq}

\begin{rmq}\label{R:precomposing transfinite}
Precomposing a transfinite mutation with finite sequences of completed mutations gives rise to a transfinite mutation: Let $T$ and $T'$ be triangulations of $\IGC$ such that there exists a finite $T$-admissible composition of completed mutations $\overline{\alpha} = (\underline{\alpha_1}, \ldots, \underline{\alpha_n})$ with $\mu_{\overline{\alpha}}(T) = T'$. If there is a $T'$-admissible sequence of completed mutations $\overline{\beta}$ with $\mu_{\overline{\beta}}(T')=T''$ then the sequence $\overline{\gamma} = (\underline{\alpha_1}, \ldots, \underline{\alpha_n}, \overline{\beta})$ is a $T$-admissible sequence of completed mutations with $\mu_{\overline{\gamma}}(T) = T''$.
\end{rmq}

Similarly, postcomposing a transfinite mutation with finite sequences of completed mutations gives rise to a transfinite mutation. To show this, the following results are useful.

\begin{lem}\label{L:commute completed}
 Let $T$ be a triangulation of $\IGC$ and let $\delta \in T$ be mutable and $S(\delta)$ be the quadrilateral in $T$ with diagonal $\delta$. If $\underline{\theta}$ is a $T$-admissible sequence which leaves all arcs in $\{\delta\} \cup S(\delta)$ untouched then $\underline{\theta}$ is $\mu_{\delta}(T)$-admissible with $\overline{\mu_{\underline{\theta}}}(\mu_{\delta}(T)) = \mu_{\delta}(\overline{\mu_{\underline{\theta}}}(T))$.
\end{lem}

\begin{proof}
 By Lemma \ref{L:untouched} the sequence $\underline{\theta}$ is $T$-admissible with $\mu_{\underline{\theta}}(\mu_{\delta}(T)) = \mu_{\delta}(\mu_{\underline{\theta}}(T))$. Let $\delta' \neq \delta$ be the other diagonal in the quadrilateral $S(\delta)$. We have
  \begin{eqnarray*}
   \overline{\mu_{\underline{\theta}}}(\mu_{\delta}(T))
			&=& \mu_{\underline{\theta}}(\mu_{\delta}(T)) \cup \mathcal{P}(\mu_{\underline{\theta}}(\mu_{\delta}(T))) \cup \tilde{\mathcal{A}}(\mu_{\underline{\theta}}(\mu_{\delta}(T)))\\
			&=& \mu_\delta (\mu_{\underline{\theta}}(T)) \cup \mathcal{P}(\mu_\delta (\mu_{\underline{\theta}}(T))) \cup \tilde{\mathcal{A}}(\mu_\delta (\mu_{\underline{\theta}}(T)))\\
			&=& ((\mu_{\underline{\theta}}(T) \cup \{\delta'\}) \setminus \{\delta\}) \cup \mathcal{P}(\mu_{\underline{\theta}}(T)) \cup \tilde{\mathcal{A}}(\mu_{\underline{\theta}}(T))\\
			&=& ((\overline{\mu_{\underline{\theta}}}(T) \cup \{\delta'\}) \setminus \{\delta\})
			= \mu_{\delta}(\overline{\mu_{\underline{\theta}}}(T)).
  \end{eqnarray*}
\end{proof}

\begin{prop}\label{P:precomposing transfinite}
 Let $T$ be a triangulation of $\IGC$ and let $\overline{\theta} = (\underline{\theta}^i)_{i \in I}$
 be a $T$-admissible composition of completed mutations such that $\mu_{\overline{\theta}}(T)=T'$ is a triangulation. If $\delta \in T'$ is mutable, then there exists an $r \in I$ such that for all $l \geq r$ the sequence
 \[
  \overline{\theta} \cup_l \{\delta\} = (\underline{\theta}^1, \ldots, \underline{\theta}^{l-1}, (\delta), \underline{\theta}^l, \underline{\theta}^{l+1}, \ldots) 
 \]
 is a $T$-admissible composition of completed mutations with $\mu_{\overline{\theta} \cup_l (\delta)}(T) = \mu_{\delta}(T')$.
\end{prop}

\begin{proof}
 Set $T_1 = T$ and for $i \in I$ set $T_{i+1} = \overline{\mu_{\theta^i}}(T_i)$. Since $\delta \in T'$ is exchangeable, we have $S(\delta) \cup \{\delta\} \subseteq T'$ and thus there exists an $r \in I$ such that $S(\delta) \cup \{\delta\} \subseteq T_r$ and for all $k \geq r$ the sequence $\underline{\theta}^k$ leaves all arcs in $S(\delta) \cup \{\delta\}$ untouched. Pick $l \geq r$ and consider the sequence $\overline{\theta} \cup_l \{\delta\}$. We first show that this is a $T$-admissible composition of admissible sequence. This is a direct consequence of the following three observations.
 
 Observation 1: For all $1 \leq i \leq l-1$ the sequence $\underline{\theta}^i$ is $T_i$-admissible.
 
 Observation 2: Since we have $S(\delta) \cup \{\delta\} \subseteq T_l$, the arc $\delta$ is mutable in $T_l$. Therefore the sequence $(\delta)$ is $T_l$-admissible.
 
 Observation 3: For $k \geq l$ set $\tilde{T}_k = \mu_{\delta}(T_k)$. By Lemma \ref{L:commute completed} the sequence $\underline{\theta}_k$ is $\tilde{T}_k$-admissible and we have
 \[
  \overline{\mu_{\underline{\theta}^k}}(\tilde{T}_k) = \overline{\mu_{\underline{\theta}^k}}(\mu_{\delta}(T_k)) = \mu_{\delta}(\overline{\mu_{\underline{\theta}^k}}(T_k)) = \mu_{\delta}(T_{k+1}) = \tilde{T}_{k+1}.
 \]
 
We notice that $\delta' \in \tilde{T}_l$ and for all $i \geq l$ the sequence $\underline{\theta}^i$ leaves $\delta'$ untouched (since it is $T_i$-admissible and it leaves $S(\delta) \cup \{\delta\} \subseteq T_i$ untouched). We obtain
 \begin{eqnarray*}
 \mu_{\overline{\theta} \cup_l (\delta)}(T) &=& \bigcup_{i \geq l}\{\gamma \in \tilde{T}_i \mid \underline{\theta}^k \; \text{leaves $\gamma$ untouched for all}\; k \geq i\}\\
 &=& \bigcup_{i \geq l}\{\gamma \in \tilde{T}_i \setminus \{\delta'\} \mid \underline{\theta}^k \; \text{leaves $\gamma$ untouched for all}\; k \geq i\} \cup \{\delta'\}\\
 &=&  \bigcup_{i \geq l}\{\gamma \in T_i \setminus \{\delta\} \mid \underline{\theta}^k \; \text{leaves $\gamma$ untouched for all}\; k \geq i\} \cup \{\delta'\}\\
 &=& (T' \setminus \{\delta\}) \cup \{\delta'\} = \mu_{\delta}(T'),
 \end{eqnarray*}
which proves the claim.
                                                                                                                                                                                                                                                                                                                                                                                                                                                                                                                                                                                                                                                                                                                                                                           
\end{proof}

\begin{prop}\label{P:transitivity transfinite}
	Let $T$ and $T'$ be triangulations of $\IGC$ such that there exists a $T$-admissible composition of completed mutations $\overline{\alpha}$ with $\mu_{\overline{\alpha}}(T) = T'$. If $\underline{\beta}$ is a $T'$-admissible sequence with $\overline{\mu_{\underline{\beta}}}(T') = T''$, then there exists a $T$-admissible composition of completed mutations  $\overline{\gamma}$ with $\mu_{\overline{\gamma}}(T) = T''$.
\end{prop}

\begin{proof}
 Let $\overline{\alpha} = (\underline{\alpha}^i)_{i \in I_{\alpha}}$ and let $\underline{\beta} = (\beta_i)_{i \in I_\beta}$. The statement is trivial if $I_\alpha$ is finite, and follows by iteratively applying Proposition \ref{P:precomposing transfinite} if $I_\beta$ is finite. We therefore assume that $I_\alpha = I_\beta = \ZZ_{>0}$. Iteratively applying Proposition \ref{P:precomposing transfinite} we obtain a strictly increasing sequence $(l_i)_{i \in \ZZ_{>0}}$ such that $\overline{\alpha} \cup_{l_1} (\beta_1)$ is a $T$-admissible composition, for all $i \geq 2$ the sequence
 \[
  \overline{\alpha} \cup (\beta_1, \ldots, \beta_i) = (\overline{\alpha} \cup (\beta_1, \ldots, \beta_{i-1})) \cup_{l_i} (\beta_i)
 \]
 is as well, and we have
 \[
  \mu_{\overline{\alpha} \cup (\beta_1, \ldots, \beta_i)}(T) = \mu_{\beta_i} \circ \ldots \circ \mu_{\beta_1}(\mu_{\overline{\alpha}}(T)).
 \]
 We define a sequence $\overline{\gamma} = (\underline{\gamma}^i)$ with
 \[
  \underline{\gamma}^i = \begin{cases}
              \underline{\alpha}^i \; \text{if} \; i \notin \{l_j \mid j \in \ZZ_{>0}\}\\
              (\beta_j) \cup_1 (\underline{\alpha}^{l_j}) \; \text{if} \; i = l_j.
             \end{cases}
 \]
  Here, $(\beta_j) \cup_1 (\underline{\alpha}^{l_j})$ is the sequence we obtain by precomposing the sequence $\underline{\alpha}^{l_j}$ by $(\beta_j)$. Clearly, the sequence $\overline{\gamma}$ is a $T$-admissible composition of completed mutations.

 Set now $ T_1 = \tilde{T}_1 = T$ and for $i \geq 1$ set
 \[
   T_{i+1} = \overline{\mu_{\underline{\alpha}^i}}(T_i) \; \; \text{and} \; \;   \tilde{T}_{i+1} = \overline{\mu_{\underline{\gamma}^i}}(\tilde{T}_i).                                                        
 \]
 Schematically we have the diagram
  \begin{align}\label{diagram}
  \xymatrix@C=2em{ \tilde{T}_1 \ar[r]^-{\overline{\mu_{\underline{\gamma}_1}}} &  
  	    \ldots \ar[r]^-{\overline{\mu_{\underline{\gamma}_{l_1-1}}}} & 
	    \tilde{T}_{l_1}  \ar[r]^-{\overline{\mu_{\underline{\gamma}_{l_1}}}} & 
	    \tilde{T}_{l_1+1} \ar[r]^-{\overline{\mu_{\underline{\gamma}_{l_1+1}}}} & 
	    \ldots \ar[r]^{\overline{\mu_{\underline{\gamma}_{l_i-1}}}} & 
	    \tilde{T}_{l_i} \ar[r]^-{\overline{\mu_{\underline{\gamma}_{l_i}}}} & 
	    \tilde{T}_{l_i+1} \\ T_1 \ar[r]^-{\overline{\mu_{\underline{\alpha}_1}}} \ar@{=}[u] & 
	    \ldots  \ar[r]^-{\overline{\mu_{\underline{\alpha}_{l_1-1}}}} & 
	    {T}_{l_1} \ar[u]^{\mu_{\beta_1}} \ar[r]^-{\overline{\mu_{\underline{\alpha}_{l_1}}}} & 
	    {T}_{l_1+1} \ar[u]^{\mu_{\beta_1}} \ar[r]^-{\overline{\mu_{\underline{\alpha}_{l_1+1}}}} & 
	    \ldots \ar[r]^{\overline{\mu_{\underline{\alpha}_{l_i-1}}}} & 
	    {T}_{l_i} \ar[u]^{\mu_{\beta_i} \circ \ldots \circ \mu_{\beta_1}} \ar[r]^-{\overline{\mu_{\underline{\alpha}_{l_i}}}} & 
	    T_{l_i+1} \ar[u]_-{\mu_{\beta_i} \circ \ldots \circ \mu_{\beta_1}} }
  \end{align}
  where each of the squares commutes. That is, for each $i \in \ZZ_{>0}$ there exists a $j \in \ZZ_{>0}$ with $l_j \leq i < l_{j+1}$ and we have
  \[
   \mu_{\beta_j} \circ \ldots \circ \mu_{\beta_1}(T_i) = \tilde{T}_i.
  \]
 
 We now show that $T'' = \mu_{\overline{\gamma}}(T)$, and we start by showing that $T'' \subseteq \mu_{\overline{\gamma}}(T)$. Let thus $\delta'' \in T'' = \mu_{\underline{\beta}}(\mu_{\overline{\alpha}}(T))$. There exists a $\delta' \in T' = \mu_{\overline{\alpha}}(T)$ such that $\delta'' = \mu^{T'}_{\underline{\beta}}(\delta')$ and an $m \geq 1$ such that $(\beta_i)_{i > m}$ leaves $\delta''$ untouched. Mutating $T'$ along the sequence $(\beta_i)_{1 \leq i \leq m}$ only changes finitely many arcs of $T'$ and thus there exists a finite union $P$ of finite subpolygons of $T'$ such that $\delta' \in T' \mid_P$ and such that $(\beta_i)_{1 \leq i \leq m}$ leaves all arcs in $T' \setminus (T' \mid_P)$ untouched.
 
 Since $T' = \mu_{\overline{\alpha}}(T)$, there exists an $n \geq 1$ such that $T' \mid_P \subseteq T_n$ and such that for all $k \geq n$ the sequence $\underline{\alpha}_k$ leaves all arcs in $T' \mid_P$ untouched. For all $k \geq n$ we obtain
 \[
  \mu^{T_k}_{\underline{\beta}}(\delta') = \mu_{\beta_m} \circ \ldots \circ \mu^{T_k}_{\beta_1}(\delta') = \delta''.
 \]
 
 Set $M = \max\{l_m, n\}$. We have $l_j \leq M < l_{j+1}$ for some $j \geq m$ and
 \[
 \delta'' = \mu^{T_M}_{\underline{\beta}}(\delta') = \mu_{\beta_j} \circ \ldots \circ \mu^{T_M}_{\beta_1}(\delta'),
 \]
 which lies in $\mu_{\beta_j} \circ \ldots \circ \mu_{\beta_1}(T_M) = \tilde{T}_M$. Furthermore, since for all $k \geq M$ the sequence $\underline{\alpha}_k$ leaves $\delta''$ untouched and the sequence $(\beta_i)_{i > m}$ leaves $\delta''$ untouched we also get that for all $k \geq M$ the sequence $\underline{\gamma}_k$ leaves $\delta''$ untouched. It follows that $\delta'' \in \mu_{\overline{\gamma}}(T)$ and therefore $T'' \subseteq  \mu_{\overline{\gamma}}(T)$.
 
 Since by Remark \ref{R:transf.nc} the set $\mu_{\overline{\gamma}}(T)$ consists of mutually non-crossing arcs, and since $T''$ is a triangulation, it follows that $T'' = \mu_{\overline{\gamma}}(T)$ which concludes the proof.
\end{proof}

Considering mutations along $T$-admissible compositions of completed mutations, we get a weaker form of mutation equivalence.

\begin{defi}
	Two triangulations $T$ and $T'$ of $\IGC$ are called {\em transfinitely mutation equivalent} if there exists a $T$-admissible composition of completed mutations $\overline{\theta}$ and a
	$T'$-admissible composition of completed mutations $\overline{\theta}'$ such that $\mu_{\overline{\theta}}(T) = T'$ and  $\mu_{\overline{\theta}'}(T') = T$.
\end{defi}

In the following we will show that all triangulations of $\IGC$ are transfinitely mutation equivalent. We start with a useful observations.

\begin{prop}\label{P:subgraph}
 Consider the graph $\mathbf{G}$ which has as vertices strong mutation equivalence classes of triangulations of $\IGC$ and whose arrows are given by the following data: Assume $[T]$ and $[T']$ are two distinct strong mutation equivalence classes.
 \begin{itemize}
  \item If for any $t \in [T]$ and any $t' \in [T]$ there exists a $t$-admissible sequence $\underline{\theta}$ with $\mu_{\underline{\theta}}(t) = t'$ then we draw a solid arrow.
  \item If for any $t \in [T]$ and any $t' \in [T]$ there exists a $t$-admissible composition of completed mutations $\overline{\theta} = (\underline{\theta}^i)_{i \in I}$ where $I$ is finite, and we have $\mu_{\overline{\theta}}(t) = t'$ then we draw a dashed arrow.
  \item If for any $t \in [T]$ and any $t' \in [T]$ there exists a $t$-admissible composition of completed mutations $\overline{\theta} = (\underline{\theta}^i)_{i \in I}$ where $I$ is infinite, and we have $\mu_{\overline{\theta}}(t) = t'$ then we draw a dotted arrow.
 \end{itemize}
 Then for any $a' \leq a \leq b \leq b'$ with $a, a', b, b' \in \ZZ$ diagram in Figure \ref{fig:subgraph} is a subgraph of $\mathbf{G}$.
\end{prop}

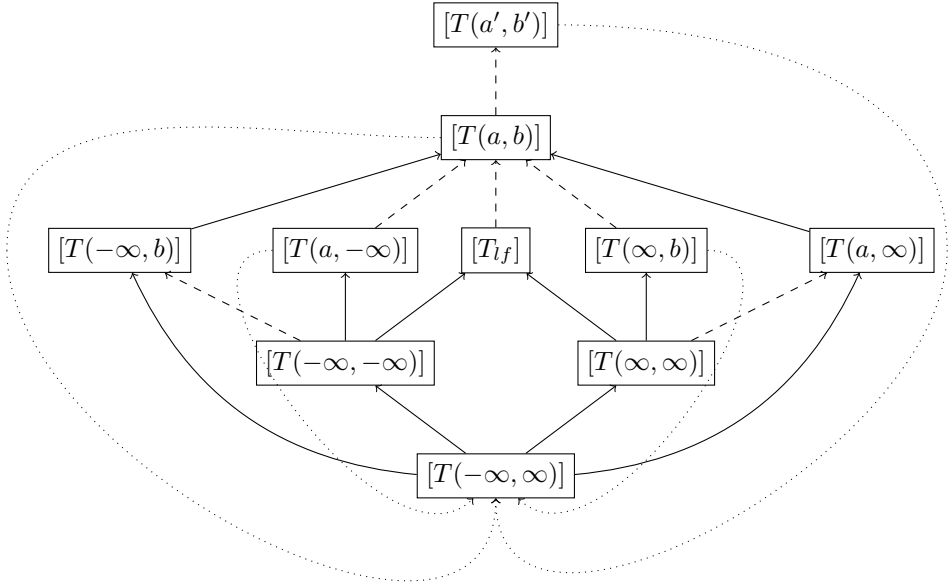
\begin{figure}
 \begin{center}
 \begin{tikzpicture}
  \node[draw, rectangle] (0) at (0,0) {$[T(-\infty, \infty)]$};
  \node[draw, rectangle] (1a) at (-2,1.5) {$[T(-\infty,-\infty)]$};
  \node[draw, rectangle] (1b) at (2,1.5) {$[T(\infty,\infty)]$};
  \node[draw, rectangle] (2a) at (-5,3) {$[T(-\infty,b)]$};
  \node[draw, rectangle] (2b) at (5,3) {$[T(a, \infty)]$};
  \node[draw, rectangle] (3a) at (-2,3) {$[T(a, -\infty)]$};
  \node[draw, rectangle] (3b) at (2,3) {$[T(\infty,b)]$};
  \node[draw, rectangle] (4) at (0,3) {$[T_{lf}]$};
  \node[draw, rectangle] (5) at (0,4.5) {$[T(a,b)]$};
  \node[draw, rectangle] (6) at (0,6) {$[T(a',b')]$};
  
  \draw[->] (0) -- (1a);
  \draw[->] (0) -- (1b);
  \draw[->] (1a) -- (3a);
  \draw[->] (1b) -- (3b);
  \draw[->] (1a) -- (4);
  \draw[->] (1b) -- (4);
  \draw[->] (0) edge[bend left] (2a);
  \draw[->] (0) edge[bend right] (2b);
  \draw[->] (2a) -- (5);
  \draw[->] (2b) -- (5);
  
  \draw[dashed, ->] (3a) -- (5);
  \draw[dashed, ->] (4) -- (5);
  \draw[dashed, ->] (3b) -- (5);
  \draw[dashed, ->] (1a) -- (2a);
  \draw[dashed, ->] (1b) -- (2b);
  \draw[dashed, ->] (5) -- (6);
  
  \draw[dotted] (6) [out=0, in=90] to (6,3);
  \draw[dotted, ->] (6,3) [out=270, in=270] to (0);
  \draw[dotted] (5) [out=180, in=90] to (-6.5,3);
  \draw[dotted, ->] (-6.5,3) [out=270, in=270] to (0);
  
  \draw[dotted,->] (3a) [out=180,in=225] to (0);
  \draw[dotted,->] (3b) [out=0,in=315] to (0);
  
 \end{tikzpicture}
 \end{center}
 \caption{A subgraph of $\mathbf{G}$}\label{fig:subgraph}
 \end{figure}

\begin{proof}
 Denote by $\mathbf{G}'$ the graph drawn in Proposition \ref{P:subgraph}. The existence of the 
 solid arrows 
 in $\mathbf{G}'$ follows from Proposition \ref{P:Hasse diagram}. 
 
 To show the existence of the dashed arrows, we first note the following: Assume for $i = 1,2$ the triangulations $T_i$ and $T'_i$ are strongly mutation equivalent and there exists a $T_1$-admissible composition of completed mutations $\overline{\theta} = (\underline{\theta}^i)_{i \in I}$ such that $\mu_{\overline{\theta}}(T_1) = T_2$ and such that $I$ is finite. Then there exists a $T'_1$-admissible sequence of completed mutations $\overline{\theta}' = (\underline{\theta}'_i)_{i \in I'}$ with $\mu_{\overline{\theta}'}(T'_1) = T'_2$ and such that $I'$ is finite: Indeed, since $T_i$ is strongly mutation equivalent to $T'_i$ for $i = 1,2$, there exists a $T'_1$-admissible sequence $\underline{\alpha}_1$ with $\mu_{\underline{\alpha}_1}(T_1) = T'_1$ and a $T_2$-admissible sequence $\underline{\alpha}_2$ with $\mu_{\underline{\alpha}_2}(T'_2) = T_2$. Setting $\overline{\theta}' = (\underline{\alpha}_1, \overline{\theta}, \underline{\alpha}_2)$ yields the desired sequence.
 
 To show the existence of a dashed arrow from $[T]$ to $[T']$, it is thus enough to show that there exist triangulations $t \in [T]$ and $t' \in [T]$ and a finite sequence of completed mutations from $t$ to $t'$. We use the notation from Remark \ref{R:representants}. 
 
 The arrows $[T(\infty,\infty)] \to [T(a,\infty)]$, $[T(-\infty,-\infty)] \to T(-\infty,b)$ and $[T(\infty,b)] \to [T(a,b)]$: Setting $\underline{\theta}_1 = (\pi_{a-i})_{i \geq 1}$ and $\underline{\theta}_2 = (\alpha_{b+i})_{i \geq 1}$ yields
 \[
  \overline{\mu_{\underline{\theta}_1}}(\overline{t}(\infty,\infty)) = \overline{t}(a, \infty), \; \; \overline{\mu_{\underline{\theta}_2}}(\overline{t}(-\infty,-\infty)) = \overline{t}(-\infty,b) \; \; \text{and} \;\; \overline{\mu_{\underline{\theta_1}}}(\overline{t}(\infty,b)) = \overline{t}(a,b).
 \]
 
 The arrow $[T(a,-\infty)] \to [T(a,b)]$: Note that we have
 \begin{eqnarray*}
  \overline{\mu_{\underline{\theta}_2}}(\overline{t}(a,-\infty)) &=& \{(k,a) \mid k \in \mathbb{Z}_{\leq a-2}\} \; \cup \; \{\alpha_k \mid a \leq k \leq b\}  \\ &&  \cup \{(b,k) \mid k \in \mathbb{Z}_{\geq b+2}\} \cup \; \{\pi_b\},
 \end{eqnarray*}
 which lies in $[T(a,b)]$.
 
 The arrow $[T_{lf}] \to [T(a,b)]$: Pick an $l \in \ZZ$ with $a \leq l \leq b$ and consider the triangulation 
 \[
  t_{lf}(l) = \{(l-k,l+k) \mid k \in \mathbb{Z}_{>0}\} \cup \{(l-k,l+k+1) \mid k \in \mathbb{Z}_{>0}\} \in [T_{lf}].
 \]
 First consider the $t_{lf}(l)$-admissible sequence $\underline{\alpha} = ((l-1,l+1), (l-2,l+1), (l-2,l+2), (l-3,l+2), \ldots, (l-i,l+i), (l-(i+1),l+i), \ldots)$. We have
 \[
  \mu_{\underline{\alpha}}(t_{lf}(l)) = \overline{t}(l,l).
 \]
 Iteratively applying Lemma \ref{L:pushing} and pushing the left fountain at $l$ to the left and the right fountain at $l$ to the right, we have a get a $T$-admissible composition of completed mutations
 \[
  \overline{\theta}_3 = (\underline{\theta}_3^i)_{i \in \{1, \ldots, b-a+1\}} = (\underline{\alpha}, \underline{l}^-, \underline{(l-1)}^-, \ldots, \underline{(a+1)}^-, \underline{l}^+, \underline{(l+1)}^+, \ldots, \underline{(b-1)}^+)
 \]
 with $\mu_{\underline{i}^-}(\overline{t}(i,l)) = \overline{t}(i-1,l)$ for all $l \geq i \geq a+1$ and $\mu_{\underline{i}^+}(\overline{t}(a,i)) = \overline{t}(a,i+1)$ for all $l \leq i \leq b-1$. We obtain $\mu_{\overline{\theta}_3}(t_{lf}(l)) = \overline{t}(a,b) \in [T(a,b)]$.
 
 The arrow $[T(a,b)] \to [T(a',b')]$: Similarly to the above considerations, this follows by iteratively pushing the left fountain at $a$ to the left and the right fountain at $b'$ to the right.

 Finally, we show the existence of the dotted arrows. By Remark \ref{R:precomposing transfinite} and Proposition \ref{P:transitivity transfinite}, to show that there is a dotted arrow from $[T]$ to $[T']$ it suffices to show that there exists a transfinite sequence of mutations from one representant of $[T]$ to one representant of $[T']$.
 
 The arrow $[T(a,b)] \to [T(-\infty,\infty)]$ for any $a \leq b$: Consider the $\overline{t}(a,b)$-admissible composition of completed mutations
 \[
  \overline{\theta}_4 = (\underline{(a-i)}^-, \underline{(b+i)}^+)_{i \geq 0},
 \]
 where we pick $\underline{(a-i)}^-$ and $\underline{(b+i)}^+$ according to Lemma \ref{L:pushing} such that for $k \geq 0$ we have $\overline{\mu_{\underline{(a-k)}^-}}(\overline{t}(a-k,b+k)) = \overline{t}(a-k-1,b+k)$ and $\overline{\mu_{\underline{(b+k)}^-}}(\overline{t}(a-k-1,b+k)) = \overline{t}(a-k-1,b+k+1)$. We obtain 
 \[
  \mu_{\overline{\theta}_4}(\overline{t}(a,b)) = \{\alpha_k \mid k \leq a\} \cup \{\beta_k \mid k \geq a\} \in [T(-\infty,\infty)].
 \]
 
 The arrows $[T(a,\infty)] \to T(-\infty,\infty)$ and $[T(-\infty,b)] \to T(-\infty,\infty)$: With $\underline{(a-i)}^-$ and $\underline{(b+i)}^+$ as above, we set $\overline{\theta}_5 = (\underline{(a-i)}^-)_{i \geq 0}$ and $\overline{\theta}_6 = (\underline{(b+i)}^+)$ and obtain 
 \[
  \mu_{\overline{\theta}_5}(\overline{t}(a,\infty)) = \{\alpha_k \mid k \leq a\} \cup \{\beta_k \mid k \geq a\}  \in [T(-\infty,\infty)]
 \]
 and
 \[
  \mu_{\overline{\theta}_6}(\overline{t}(-\infty,b)) = \{\alpha_k \mid k \leq b\} \cup \{\beta_k \mid k \geq b\} \in [T(-\infty,\infty)].
 \]

 \end{proof}

\begin{theorem}\label{T:all-transf-equiv}
	All triangulations of $\IGC$ are transfinitely mutation equivalent.
\end{theorem}

\begin{proof}
 Let $T$ and $T'$ be two triangulations of $\IGC$ and consider their strong mutation equivalence classes $[T]$ and $[T']$ respectively. Then there exists a path in the graph $\mathbf{G}'$ from Proposition \ref{P:subgraph}, and therefore in $\mathbf{G}$, of the form
 \[
  \xymatrix{[T] \ar[r]^-{\alpha_1} & [T_1] \ar[r]^-{\alpha_2} & \ldots \ar[r]^-{\alpha_l} & [T_l] \ar[r]^-{\beta} & [T'_1] \ar[r]^-{\gamma_1} & \ldots & [T'_k] \ar[r]^-{\gamma_k} & [T'],}
 \]
 with $l,k \in \ZZ_{\geq 0}$ and where the $\alpha_i$ are 
 solid or dashed arrows, the arrow $\beta$ is dotted and the arrows $\gamma_i$ are solid. That is, we have a $T$-admissible composition of completed mutations  $\overline{\alpha}=(\underline{\alpha}_i)_{i=1,\ldots,l}$ with $\mu_{\underline{\alpha}}(T) = T_l$ for some $T_l \in [T_l]$, and therefore by Remark \ref{R:precomposing transfinite} a $T$-admissible composition of completed mutations $\overline{\beta}$ with $\mu_{\overline{\beta}}(T) = T'_1$  for some $T'_1 \in [T'_1]$. By Proposition \ref{P:transitivity}, there is a $T'_1$-admissible sequence $\underline{\gamma}$ with $\mu_{\underline{\gamma}}(T'_1) = T'$ and therefore by Proposition \ref{P:transitivity transfinite} we get a $T$-admissible composition of completed mutations $\overline{\beta} \cup \underline{\gamma}$ with $\mu_{\overline{\beta} \cup \underline{\gamma}}(T) = T'$.
 
\end{proof}

\section*{Acknowledgements}
The first author was supported by the Austrian Science Fund (FWF) grants W1230, 
P 25647-N26 and P 25141-N26, the second author thanks the Swiss National Science Foundation (SNSF) for financial support through an early Postdoc.Mobility fellowship.

\bibliographystyle{plain}

\bibliography{biblioSira}

\end{document}